\documentclass[12pt]{amsart}  
\usepackage{amssymb,amsfonts,amsmath,bm}
\usepackage{epsfig}
\usepackage{subfigure}
\usepackage{bbm,color}
\usepackage{hyperref}
 \usepackage{mathptmx}     

\usepackage{pstricks}
 \usepackage{pgfplots}
 
\newcommand{\rme}{{\rm e}}
\newcommand{\rmd}{{\rm d}}
\newcommand{\la}{{\lambda}}
\newcommand{\ep}{{\varepsilon}}
\newcommand{\1}{\mathbbm{1}}
\newcommand{\pd}{{\partial}}

\newcommand{\PP}{{\mathbb P}}

\newcommand{\EE}{{\mathbb E}}
\newcommand{\TT}{{\mathbb T}}

\renewcommand{\theequation}{\thesection.\arabic{equation}}

\newtheorem{theo}{Theorem}[section]

\newtheorem{rem}{Remark}[section]
\newtheorem{lem}{Lemma}[section]

\begin{document}

\title{Ornstein-Uhlenbeck processes of bounded variation }

\author{
Nikita Ratanov}\thanks{Published in  \emph{Methodology
and Computing in Applied Probability}, (2020) DOI 10.1007/s11009-020-09794-x}
   \address{           Universidad del Rosario, Cl.\,12c, No.~4-69, Bogot\'a, Colombia \\        
   {nikita.ratanov@urosario.edu.co}           %  \\
}
\date{}

\maketitle

\begin{abstract}
Ornstein-Uhlenbeck process of bounded variation
 is introduced as a solution of an analogue of the Langevin equation 
with an integrated telegraph process replacing a Brownian motion.
There is an interval $I$ such that the process starting 
from the internal point of $I$ always remains within $I$.
Starting outside, this process a. s. reaches this interval in a finite time.
The distribution of the time for which the process falls into this interval is obtained explicitly.

The certain formulae for the mean and the variance of this process are obtained 
on the basis of the joint distribution of the telegraph process and its integrated copy.

Under Kac's rescaling, the limit process is identified as the classical Ornstein-Uhlenbeck process.

\end{abstract}
\date{}
\maketitle

\section{Introduction}

For a long time by various reasons, different finite-velocity diffusion models have been 
studied as a substitute for classical diffusion,  
which is described by a parabolic equation with infinitely fast propagation.
 The main model represents
motions performed by a particle moving along a line at a finite velocity 
and changing directions after exponentially distributed holding times, see \cite{cattaneo}.
The corresponding random process of particle's position 
is called an integrated \emph{telegraph process}.
The distribution of this process is described by the damped wave equation 
(hyperbolic diffusion equation, the so-called telegraph equation).

The one-dimensional version of the telegraph process $\TT(t),\;t\geq0,$ 
with two alternating regimes is well studied, 
starting with the seminal works of M.Kac, see \cite{Kac}. 
This theory has a huge literature, see, for example, surveys in \cite{KR,Zacks}.

To introduce the telegraph process, 
we consider a two-state Markov process $\ep=\ep(t)\in\{0, 1\},$ 
defined on the complete filtered probability space
$(\Omega,\;\mathcal F,\;\mathcal F_t,\;\PP)$.
Process $\ep$ is determined by two positive switching parameters $\la_0, \la_1:$
\[
\PP\{\ep(t+\rmd t)=i~|~\ep(t)=i\}=1-\la_i\rmd t+o(\rmd t),\qquad \rmd t\to0,\;i\in\{0, 1\}.
\]
We define the (integrated) telegraph process by
\[
\TT(t)=\int_0^ta_{\ep(s)}\rmd s,
\]
where $a_0, a_1$ are constants;
$\TT(t)$ is the position of a particle moving in a line
with velocities $a_0$ and $a_1$ alternating at random times.
%
%\begin{color}{magenta}
Since $\ep$ is the time-homogeneous Markov process,
the (conditional) distribution of $\TT(t)-\TT(s)=\int_{s}^{t}a_{\ep(s')}\rmd s'$ 
 and $\TT(t-s)=\int_0^{t-s}a_{\ep(s')}\rmd s'$ are identical for any $s, t,\;0\le s<t,$
precisely, the following identity in law holds:
\begin{equation}
\label{eq:persistence}
\left[\TT(t)-\TT(s)=\left.\int_s^ta_{\ep(s')}\rmd s'~\right|~\mathcal F_s\right]\stackrel{D}{=}
\left[\TT(t-s)=\left.\int_0^{t-s}a_{\ep(s')}\rmd s'~\right|~\ep(0)\right],
\end{equation}
see e.g. \cite{KR}.
%\end{color} 

The Gaussian Ornstein-Uhlenbeck process $X^{OU}$ 
is another class of processes we are interested in. 
This process can be defined as the solution to the stochastic differential equation
\begin{equation}\label{eq:Langevin} 
\rmd X^{OU}(t)=-\gamma X^{OU}(t)\rmd t+\sigma\rmd W(t),\qquad t>0,
\end{equation}
where $W=W(t)$ is the standard Brownian motion and $\gamma>0.$
This model is used in various application areas 
as an alternative to Brownian motion with an average tendency to return, see \cite{Langevin,Maller}.
Let me mention here only two of these application areas.
The Va\v{s}\'i\v{c}ek interest rate model, \cite{vasicek}, is the most famous 
financial application of this process. 
The same processes are also widely used for neuronal modelling, see e. g. \cite{pirozzi}.
Similar application areas have telegraph processes:
for financial applications see e. g. \cite{DiC14,KR,Quant}; %\begin{color}{magenta}
the first steps in the neuronal modelling %\end{color}
based on the telegraph process are presented by \cite{MBE} and \cite[Section 2.3.2]{genadot}.

In this paper, we study the Ornstein-Uhlenbeck process of bounded variation, which is determined 
by the version of Langevin equation \eqref{eq:Langevin} 
when the Brownian motion $W$ is replaced by
telegraph process $\TT.$ 
More precisely, let $X=X(t)$ be a stochastic process defined by the equation
\[
\rmd X(t)=-\gamma_{\ep(t)}X(t)\rmd t+\rmd\TT(t),\qquad t>0,
\]
where $\TT(t)$ is the telegraph process based on the Markov process $\ep.$ 
Since Kac's telegraph process  $\TT$  
is used instead of the Wiener process in the usual Langevin equation, \cite{Langevin},
this equation can be called the Kac-Langevin equation. 
The latter stochastic equation is equivalent to an integral equation of the following form,
\begin{equation}
\label{eq:KacLangevin}
X(t)=x-\int_0^t\gamma_{\ep(s)}X(s)\rmd s+\TT(t),\qquad t>0.
\end{equation}
Here $x=X(0)$ is the starting point of the process $X.$
To the best of my knowledge, such a modification of the 
Langevin equation has not been studied before.

The detailed problem settings are presented by Section \ref{sec3}.
Not surprisingly, the analysis of the distribution of $X(t)$ is not as simple as for 
the Gaussian process $X^{OU}=X^{OU}(t)$. 
The first peculiarity is the following. 

If the starting point $x$ is in the interval $I=(a_1/\gamma_1,\;a_0/\gamma_0),$ $x\in I,$
the process $X(t)$ remains inside the band, that is $X(t)\in(a_1/\gamma_1,\;a_0/\gamma_0),\;t\geq0.$ 
In contrary, if the process $X$ starts from outside of $I,\;x\notin[a_1/\gamma_1,\;a_0/\gamma_0],$
the process reaches $I$ a. s. in finite time (here, we assume that $a_1/\gamma_1<a_0/\gamma_0$).

%\begin{color}{magenta}
The main goal of this paper (see Section \ref{sec4})
is to study the distribution of time over which the process $X = X (t),$ 
starting from the outside of interval $I,$  falls into $I.$ %\end{color}
This problem is associated with the first passage time of the telegraph process, 
 which has been intensively studied recently, see, e. g. \cite{DiC18,MCAP19,MBE,Zacks2004,Zacks}.

To analyse Ornstein-Uhlenbeck process of bounded variation, \eqref{eq:KacLangevin},
 we need to study the properties of the stochastic integral
\begin{equation}\label{eq:stochint}
\mathcal I(t)=\int_0^t\varphi(s)\rmd\TT(s)
=\int_0^t\varphi(s)a_{\ep(s)}\rmd s,\qquad t>0.
\end{equation}
Since $\varphi(\cdot)a_{\ep(\cdot)}$ a. s. has a finite number of discontinuities on $[0, t],$ 
the integral in \eqref{eq:stochint} can be considered as a \emph{pathwise} Riemann integral.
The stochastic process $\mathcal I(t),\;t>0,$ can be considered 
as a generalised telegraph process with two \emph{time-varying} velocity patterns,
$a_0\varphi(t)$ and $a_1\varphi(t),$ alternating after exponentially distributed holding times.
The rectifiable version of such a process has been studied in detail by \cite{JAP2019},
but $\mathcal I=\mathcal I(t),$ defined by \eqref{eq:stochint}, does not belong to this class.

The distribution of $X(t)$ looks much more sophisticated than the Gaussian distribution 
of the Ornstein-Uhlenbeck process $X^{OU}(t)$. Sections \ref{sec5} and \ref{sec6} take
only a few simple first steps for this analysis. 

For completeness, in the Appendix we recall some modern results on
telegraph processes, including explicit formulas for the joint distribution of $X(t)$ and $\ep(t),$
which have never been published before.

%\begin{color}{magenta}
%Without loss of generality,  in this paper we consider only the symmetric
%case, $a_0=-a_1=a,\;a>0.$
%\end{color}
%%%%%%%%%%%%%%%%%%%%%%%%%%%%%%%%%%%%%%%%%%%%%%%
%%%%%%%%%%%%%%%%%%%%%%%%%%%%%%%%%%%%%%%%%%%%%%%

%%%%%%%%%%%%%%%%%%%%%%%%%%%%%%%%%%%%%%%%%%%
%%%%%%%%%%%%%%%%%%%%%%%%%%%%%%%%%%%%%%%%%%%
 \section{The problem setting}\label{sec3}
 %Kac-Langevin equation and Ornstein-Uhlenbeck processes}
\setcounter{equation}{0}

 We study the  path-continuous random process $X=X(t),\;t\ge0,$ satisfying
the stochastic equation \eqref{eq:KacLangevin}, that is
\begin{equation}\label{eq:KL}
\rmd X(t)=\left(a_{\ep(t)}-\gamma_{\ep(t)}X(t)\right)\rmd t,\qquad t>0,
\end{equation}
with the initial condition $X(0)=x$. Recall, that $\ep=\ep(t)$ is the two-state Markov process, 
$a_0, a_1\in(-\infty, \infty)$ and $\gamma_0, \gamma_1>0,$ $a_0/\gamma_0>a_1/\gamma_1.$

After applying the usual  integrating factor technique, 
one can see that \eqref{eq:KL} is equivalent to 
\[
\rmd\left(\rme^{\Gamma(t)}X(t)\right)=\rme^{\Gamma(t)}a_{\ep(t)}\rmd t,
\]
where $\Gamma(t)=\int_0^t\gamma_{\ep(s)}\rmd s$ is the integrated telegraph process 
based on the same underlying process $\ep$ as the telegraph process 
$\TT(t)=\int_0^ta_{\ep(s)}\rmd s.$
This yields the formula for the solution of \eqref{eq:KL}:
\begin{equation}
\label{eq:solOU1}
X(t)=\rme^{-\Gamma(t)}\left(x+\int_0^t\rme^{\Gamma(s)}a_{\ep(s)}\rmd s\right)
=\rme^{-\Gamma(t)}\left(x+\int_0^t\rme^{\Gamma (s)}\rmd \TT(s)\right),
\qquad t\geq0.
\end{equation}

The process 
$X=X(t)$ can be considered as a piecewise deterministic path-continuous process 
of bounded variation which follows the two patterns,
\begin{align}
\label{eq:pattern0}
    \phi_0(x, t)& =\rme^{-\gamma_0 t}\left(x+a_0\int_0^t\rme^{\gamma_0s}\rmd s\right)
     % x\rme^{-\gamma t}+a\frac{1-\rme^{-\gamma t}}{\gamma}
      =\frac{a_0}{\gamma_0}+\left(x-\frac{a_0}{\gamma_0}\right)\rme^{-\gamma_0 t},\\
    \label{eq:pattern1}
    \phi_1(x, t)&  =\frac{a_1}{\gamma_1}+\left(x-\frac{a_1}{\gamma_1}\right)\rme^{-\gamma_1 t},\quad t\geq0,
\end{align}
alternating at Poisson times.  
Similarly defined generalisation of the telegraph process 
 were recently studied by \cite{JAP2019}; however, here the
process $X$ does not satisfy the homogeneity property 
with a common rectifying mapping, see \cite[(2.13)]{JAP2019}, which creates new difficulties.

Let $\tau=\tau^{(0)}$ ($\tau=\tau^{(1)}$)
be the first  switching time if  
$\ep(0)=0$ ($\ep(0)=1$).
The following identities of conditional distributions 
can be proved by conditioning on the first switch,
\begin{align}
\label{eq:law01}
    \left[X(t)~|~\ep(0)=0,\;X(0)=x\right]& \stackrel{D}{=}    \left[X(t-\tau)~|~\ep(0)=1,\;X(0)=\phi_0(x, \tau)\right],  \\
    \label{eq:law10}
    \left[X(t)~|~\ep(0)=1,\;X(0)=x\right]& \stackrel{D}{=}    \left[X(t-\tau)~|~\ep(0)=0,\;X(0)=\phi_1(x, \tau)\right].
\end{align}
If there are no switching to the time horizon $t$, that is $\tau>t,$
we have 
\[
  \left[X(t)~|~\ep(0)=0,\;X(0)=x\right]=\phi_0(x, t),\quad \left[X(t)~|~\ep(0)=1,\;X(0)=x\right]=\phi_1(x, t),
  \qquad\text{a.s.}
\]

Note that the mappings $t\to\phi_0(x, t)$ and $t\to\phi_1(x, t)$ satisfy semigroup property, 
\[
\lim_{t\to\infty}\phi_0(x, t)=\frac{a_0}{\gamma_0},\qquad \lim_{t\to\infty}\phi_1(x, t)=\frac{a_1}{\gamma_1}.
\]
A sample path is shown in Fig. \ref{fig1}. 

 \begin{figure}[h]
 \begin{center}
\begin{tikzpicture}[x=1.35cm,y=1.35cm][domain=0:8]   
\draw[->] (-1.2,-0.5) -- (7.5,-0.5); 
\draw[->] (-1,-1.2) -- (-1,1.85); 
\draw[-] (-0.36,-0.5) -- (-0.36,1);
%%%%%%%%%%%%%%%%%%%%%%%%
 \draw[blue,domain=-1:-0.5] 
 	plot(\x,{1+0.5*exp(-\x-1)}); 
%%%%%%%%%%%%%%%%%%
 \draw[blue,domain=-0.5:1.5] 
 plot(\x,{0.5*exp(-\x-1)-1+2*exp(-\x-0.5)}); 
 %%%%%%%%%%%%%%%%%
  \draw[blue,domain=1.5:3.5] 
 plot(\x,{1+(-2+0.5*exp(-2.5)+2*exp(-2))*exp(-\x+1.5)}); 
 %%%%%%%%%%%%%%%%%%%%%%%
 \draw[blue,domain=3.5:6] 
  plot(\x,{(1+(-2+0.5*exp(-2.5)+2*exp(-2))*exp(-4.5+2.5))*exp(-(\x-3.5))-1+exp(-\x+3.5)}); 
   \draw[blue,domain=6:7.5] 
 plot(\x,{((1+(-2+0.5*exp(-2.5)+2*exp(-2))*exp(-4.5+2.5))*exp(-(7-4.5))-1+exp(-7+4.5))*exp(-\x+6)+1-exp(-\x+6)}); 
 %%%%%%%%%%%%%%%%%%%%%%%%%%%%%%%
 \draw [dashed] (-1,1)     to  (7.5,1); 
  \draw [dashed] (-1,-1)     to  (7.5,-1); 
  \draw [dashed] (-0.5,-1)     to  (-0.5,1.25);   
  \draw [dashed] (1.5,-1)     to  (1.5,1); 
    \draw [dashed] (3.5,-1)     to  (3.5,1); 
      \draw [dashed] (6,-1)     to  (6,1); 
        \node at (-0.6,-0.6) {\scriptsize{$\tau_1$}};
                \node at (1.4,-0.6) {\scriptsize{$\tau_2$}};
                        \node at (3.4,-0.6) {\scriptsize{$\tau_3$}};
                                \node at (5.9,-0.6) {\scriptsize{$\tau_4$}};
  \node at (7.5,-0.6) {\scriptsize{$t$}};
    \node at (-1.2,1.8) {\scriptsize{$X$}};
        \node at (-1.2,1.5) {\scriptsize{$x$}};
 \node at (-1.25,1) {\scriptsize{$a_0/\gamma_0$}};
  \node at (-1.3,-1) {\scriptsize{$a_1/\gamma_1$}};
  \node at (-0.2,-0.62) {\scriptsize{$T(x)$}};

     \draw[fill]  (-0.36,-0.5) circle [radius=0.025];
\end{tikzpicture}
 \caption{The sample path of $X=X(t). 
$ {}{}}
 \label{fig1}
 \end{center}
\end{figure}
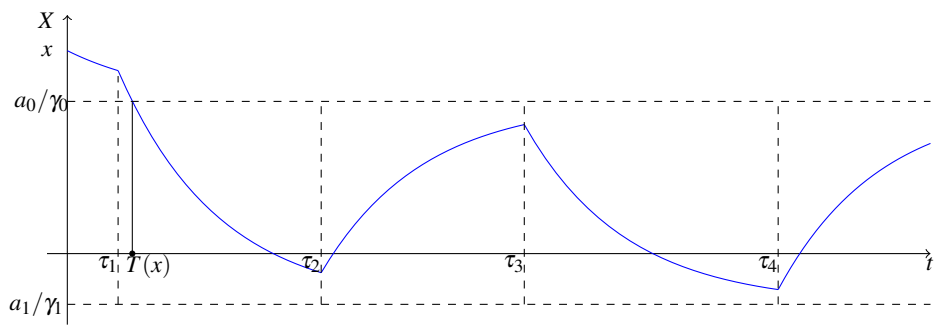 

It follows from the definition that 
if the starting point $x=X(0)$ is in the interval, $x\in I=(a_1/\gamma_1,\;a_0/\gamma_0),$ then
\begin{equation}
\label{eq:strip}
a_1/\gamma_1<\phi_0(x, t),\;\phi_1(x, t)<a_0/\gamma_0,\qquad \forall t\ge0.
\end{equation}
Further, 
if the starting point $x=X(0)$ is outside the interval $I,$ the trajectory $X=X(t)$ a. s. falls into
 $I$  and remains there for all subsequent $t.$

The distribution of this falling time is studied in the next section.

%%%%%%%%%%%%%%%%%%%%%%%%%%%%%%%%%%%%%%
\section{The falling time into the interval $I=(a_1/\gamma_1,\;a_0/\gamma_0).$}
\label{sec4}
\setcounter{equation}{0}

Let $x>a_0/\gamma_0,$ and $T(x)$
 be the time of first passage through the level $a/\gamma$ by the process $X(t),$
which starts at point $x,\;x>a_0/\gamma_0,$ 
\begin{equation}
\label{def:T}
T(x):=\inf\{t:~X(t)<a_0/\gamma_0~|~X(0)=x>a_0/\gamma_0\},
\end{equation}
see Fig.\ref{fig1}.

We denote by  $t^*(x)$ the shortest time
for crossing the level $a_0/\gamma_0$ by the process $X=X(t)$, 
which starts at $x,\; x>a_0/\gamma_0.$
This corresponds to movement only along the pattern $\phi_1(x, t)$ without switching.
Thus, $t^*(x)$  is determined by the formula
\begin{equation}
\label{def:t*}
t^*(x)=\frac{1}{\gamma_1}\log\frac{x-a_1/\gamma_1}{a_0/\gamma_0-a_1/\gamma_1}>0,
\end{equation}
which is the root of the equation, $\phi_1(x, t)=a_0/\gamma_0,$ \eqref{eq:pattern1}.
A motion only with the pattern $\phi_0(x, t),\;x>a_0/\gamma_0,$ \eqref{eq:pattern0}, 
(without switching) never crosses the level $a_0/\gamma_0$.

The distribution of $T(x)$ is supported on $\{t:~t\ge t^*(x)\},$ and can be found in the form of
the (generalised) density functions  $Q_0(t, x)$ and $Q_1(t, x)$,
\begin{equation*}
%\label{def:QQ}
\PP\{T(x)\in\rmd t~|~\ep(0)=0\}=Q_0(t, x)\rmd t,\qquad 
\PP\{T(x)\in\rmd t~|~\ep(0)=1\}= Q_1(t, x)\rmd t,
\end{equation*}
assuming that
\begin{equation*}
%\label{eq:boundary-t}
Q_0(t, x)|_{t<t^*(x)}=0,\qquad Q_1(t, x)|_{t<t^*(x)}=0.
\end{equation*}
%%%%%%%%%%%%%%%%%%%%%%%%%%%%%%%%%%%%%%%%%%%%%

Due to identities \eqref{eq:law01}-\eqref{eq:law10},
functions $Q_0(t, x)$ and $Q_1(t, x)$  
follow the system of the integral equations,
\begin{equation}
\label{eq:Q0Q1}
\left\{
\begin{aligned}
   Q_0(t, x) &=\int_0^t\la_0\rme^{-\la_0\tau}   Q_1(t-\tau, \phi_0(x, \tau))   \rmd\tau,   \\
    Q_1(t, x)&=\rme^{-\la_1t^*(x)}\delta(t-t^*(x))+  \int_0^{t^*(x)}\la_1\rme^{-\la_1\tau}Q_0(t-\tau, \phi_1(x, \tau))\rmd\tau,    
\end{aligned}
\right. 
\end{equation}
 $t>t^*(x),\;x>a_0/\gamma_0.$ Here, $\delta=\delta(\cdot)$ is Dirac's delta-function.

By definition of $T(x)$, \eqref{def:T}, equations \eqref{eq:Q0Q1} must be supplied with 
the boundary conditions
\begin{equation}
\label{eq:boundary-x}
\lim_{x\downarrow a_0/\gamma_0}Q_0(t, x)=\la_0\rme^{-\la_0t},\qquad
\lim_{x\downarrow a_0/\gamma_0}Q_1(t, x)=\delta(t).
\end{equation}
Since $\lim_{x\downarrow a_0/\gamma_0}t^*(x)=0$  (see \eqref{def:t*})
and $\lim_{x\downarrow a_0/\gamma_0}\phi_0(x, \tau)\equiv a_0/\gamma_0$ 
(see \eqref{eq:pattern0}), the same follows from the equations \eqref{eq:Q0Q1} themselves.

Let
\begin{equation*}
%\label{def:L0L1}
\mathcal L_0:=\frac{\pd}{\pd t}+(\gamma_0 x-a_0)\frac{\pd}{\pd x},\qquad
\mathcal L_1:=\frac{\pd}{\pd t}+(\gamma_1 x-a_1)\frac{\pd}{\pd x}.
\end{equation*}
Since
\begin{equation*}
%\label{eq:L0L1phi0phi1}
\mathcal L_0\left[(\gamma_0 x-a_0)\rme^{-\gamma_0 t}\right]=0,\qquad
\mathcal L_1\left[(\gamma_1 x-a_1)\rme^{-\gamma_1 t}\right]=0,
\end{equation*}
and $(\gamma_1 x-a_1)\dfrac{\rmd}{\rmd x}[t^*(x)]=1,$
we have the following identities:
\[
%\begin{aligned}
\mathcal L_0[\phi_0(x, t)]=0,\qquad
\mathcal L_1[\phi_1(x, t)]=0,\qquad
\mathcal L_1[t^*(x)]=1,
\]
see  \eqref{eq:pattern0}-\eqref{eq:pattern1} and \eqref{def:t*}. 
By applying $\mathcal L_0$ and $\mathcal L_1$ to \eqref{eq:Q0Q1} we obtain 
\begin{equation*}
%\label{eq:Q0Q1-PDE}
\left\{
\begin{aligned}
\mathcal L_0[Q_0(t, x)]=&\la_0\rme^{-\la_0t}Q_1(0, \phi_0(x, t))
-\int_0^t\la_0\rme^{-\la_0\tau}\frac{\rmd}{\rmd\tau}\Big[Q_1(t-\tau,\; \phi_0(x, \tau))\Big]\rmd\tau,
\\
\mathcal L_1[Q_1(t, x)]=&-\la_1\rme^{-\la_1t^*(x)}\delta(t-t^*(x))
+\la_1\rme^{-\la_1t^*(x)}Q_0(t-t^*(x),a/\gamma)\\
&\qquad\qquad\qquad\qquad-\int_0^{t^*(x)}\la_1\rme^{-\la_1\tau}\frac{\rmd}{\rmd\tau}\Big[Q_0(t-\tau, \phi_1(x, \tau))\Big]\rmd\tau. 
\end{aligned}
\right. 
\end{equation*}
Integrating by parts, one can see that
system \eqref{eq:Q0Q1} of the integral equations is equivalent to 
the system of the partial differential equations,
\begin{equation}
\label{eq:PDEQ}
\left\{
\begin{aligned}
    \mathcal L_0[Q_0(t, x)]&=-\la_0Q_0(t, x)+\la_0Q_1(t, x),   \\
    \mathcal L_1[Q_1(t, x)]&=\la_1Q_0(t, x)-\la_1Q_1(t, x), 
\end{aligned}
\right. 
\end{equation}
$t>t^*(x),\; x>a_0/\gamma_0,$ with the boundary conditions \eqref{eq:boundary-x}.

Consider the Laplace transforms
\begin{equation}
\label{def:laplaceq0q1}
\begin{aligned}
\widehat Q_0(q, x)=&\EE_0[\exp(-qT(x))]
=\int_0^\infty\rme^{-qt}Q_0(t, x)\rmd t,\\
\widehat Q_1(q, x)=&\EE_1[\exp(-qT(x))]
=\int_0^\infty\rme^{-qt}Q_1(t, x)\rmd t,
\end{aligned}\qquad q>0.\end{equation}
Note that $0\leq\widehat Q_i(q, x)\leq1,\;i\in\{0, 1\}.$ 
Functions $\widehat Q_0(q, x)$ and $\widehat Q_1(q, x)$ have a sense of the 
complementary cumulative distribution function 
of $\overline X_{\rme_q}=\max\limits_{0\le t\le\rme_q}X(t),$ where
$\rme_q$ is an exponentially distributed random variable, $\mathrm{Exp}(q),$
independent of $X.$
Indeed, integrating by parts in \eqref{def:laplaceq0q1}, one can see 
\[
\begin{aligned}
 \widehat Q_i(q, x)=&\int_0^\infty   \rme^{-qt}\rmd\left[\PP\{T(x)<t~|~\ep(0)=i\}\right]
 =\int_0^\infty   q\rme^{-qt}\PP\{T(x)<t~|~\ep(0)=i\}\rmd t\\
   = & \PP\{T(x)<\rme_q~|~\ep(0)=i\} =\PP\{\overline X_{\rme_q}>x~|~\ep(0)=i\}.
\end{aligned}
\]

System \eqref{eq:PDEQ} corresponds to the system of the ordinary equations,
\begin{equation}
\label{eq:q0q1}
\left\{
\begin{aligned}
    (x-a_0/\gamma_0)\frac{\rmd \widehat Q_0}{\rmd x}(q, x)=
    &-\beta_0(q)\widehat Q_0(q, x)+\beta_0(0)\widehat Q_1(q, x),   \\
     (x-a_1/\gamma_1)\frac{\rmd \widehat Q_1}{\rmd x}(q, x)=
     & \beta_1(0)\widehat Q_0(q, x) -\beta_1(q)\widehat Q_1(q, x),
\end{aligned}
\right.\qquad x>a_0/\gamma_0, 
\end{equation}
where
\begin{equation}
\label{def:beta}
\beta_0(q)=\frac{\la_0+q}{\gamma_0},\qquad \beta_1(q)=\frac{\la_1+q}{\gamma_1}.
\end{equation}
Due to \eqref{eq:boundary-x},
system \eqref{eq:q0q1}
is supplied with the boundary conditions 
\[\widehat Q_0(q, a_0/\gamma_0+)=\la_0/(\la_0+q),\qquad\widehat Q_1(q, a_0/\gamma_0+)=1.\]

Consider the series representations:
\begin{equation*}
%\label{eq:q0q1series}
\widehat Q_0(q, x)=\sum_{n=0}^\infty A_n( x-a_0/\gamma_0)^n,\qquad
\widehat Q_1(q, x)=\sum_{n=0}^\infty B_n( x-a_0/\gamma_0)^n,\qquad x>a_0/\gamma_0.
\end{equation*}
The boundary conditions give $A_0=\la_0/(\la_0+q),\;B_0=1,$ and by the system \eqref{eq:q0q1}
 we have the sequence of coupled equations 
for coefficients $A_n$ and $B_n,\;n\geq0:$
\begin{equation*}
%\label{eq:AnBn1}
\left\{
\begin{aligned}
nA_n=-\beta_0(q)&A_n+\beta_0(0)B_n,\\
nB_n+(n+1)(a_0/\gamma_0-a_1/\gamma_1)B_{n+1}
=\beta_1(0)&A_n-\beta_1(q)B_n,
\end{aligned}
\right.
\end{equation*}
which is equivalent to 
\begin{equation} \label{eq:AnBnBn} 
\left\{
\begin{aligned}
  A_n&=\frac{\beta_0(0)}{\beta_0(q)+n}B_n,   \\
    B_{n+1}&=  \frac{\beta_1(0)A_n-(\beta_1(q)+n)B_n}{(n+1)(a_0/\gamma_0-a_1/\gamma_1)}
    =\frac{\beta_0(0)\beta_1(0)-(\beta_0(q)+n)(\beta_1(q)+n)}
    {(n+1)(a_0/\gamma_0-a_1/\gamma_1)(\beta_0(q)+n)}B_n.
\end{aligned}
\right.\end{equation}
The second equation can be rewritten as
\[
B_{n+1}=-\frac{\left(b_0(q)+n\right)\left(b_1(q)+n\right)}{\beta_0(q)+n}
\cdot\frac{B_n}{(n+1)(a_0/\gamma_0-a_1/\gamma_1)},
\]
where 
\begin{equation}
\label{def:b}
b_{0, 1}=\frac12\left(\beta_0(q)+\beta_1(q)\pm
\sqrt{(\beta_0(q)-\beta_1(q))^2+4\beta_0(0)\beta_1(0)}\right).
\end{equation}

Due to the boundary conditions, $B_0=1,$
\begin{equation*}
%\label{eq:Bn}
\begin{aligned}
B_1=&\frac{b_0b_1}{\beta_0\cdot 1!}\cdot\frac{-1}{a_0/\gamma_0-a_1/\gamma_1},\\
B_2=&\frac{b_0(b_0+1)\cdot b_1(b_1+1)}{\beta_0(\beta_0+1)\cdot2!}
\cdot\left(\frac{-1}{a_0/\gamma_0-a_1/\gamma_1}\right)^2,\\
\;\ldots,\;
B_n=&\frac{(b_0)_n(b_1)_n}{(\beta_0)_n\cdot n!}
\cdot\left(\frac{-1}{a_0/\gamma_0-a_1/\gamma_1}\right)^n,
\end{aligned}
\end{equation*}
and, by the first equation of  \eqref{eq:AnBnBn}, 
\[
A_n=\frac{\beta_0(0)}{\beta_0(q)+n}\cdot\frac{(b_0)_n(b_1)_n}{(\beta_0(q))_nn!}\cdot
\left(\frac{-1}{a_0/\gamma_0-a_1/\gamma_1}\right)^n
=\frac{\la_0}{\la_0+q}\cdot\frac{(b_0)_n(b_1)_n}{(\beta_0+1)_n\cdot n!}\cdot
\left(\frac{-1}{a_0/\gamma_0-a_1/\gamma_1}\right)^n,
\]
\[\qquad n\ge0,\]
where $\beta_0=\beta_0(q),\;\beta_1=\beta_1(q)$ and 
$b_0=b_0(q),\;b_1=b_1(q)$ are defined by \eqref{def:beta} and \eqref{def:b};
$(b)_n=\Gamma(b+n)/\Gamma(b)=b(b+1)\ldots(b+n-1)$ is the Pochhammer symbol.

As a result, functions $\widehat Q_0$ and $\widehat Q_1$ are expressed by 
\begin{equation}
\label{eq:q0q1=}
\begin{aligned}
\widehat Q_0(q, x)=&\frac{\la_0}{\la_0+q}
F\left(b_0(q), b_1(q); \beta_0(q)+1; \frac{a_0/\gamma_0-x}{a_0/\gamma_0-a_1/\gamma_1}\right),\\
\widehat Q_1(q, x)=& 
F\left(b_0(q), b_1(q); \beta_0(q); \frac{a_0/\gamma_0-x}{a_0/\gamma_0-a_1/\gamma_1}\right).
\end{aligned}
\end{equation}
Here $F$ is the Gaussian hypergeometric function, defined by the series
\begin{equation}\label{eq:ghf} 
F(b_0, b_1; \beta; z)=1+\sum_{n=1}^\infty\frac{(b_0)_n(b_1)_n}{(\beta)_n\cdot n!}
\;z^n,%\qquad z=\frac{a-\gamma x}{2a},
\end{equation}
if one of the following conditions holds:
\begin{enumerate}
  \item $|z|<1;$
  \item $|z|=1$ and $\beta-b_0-b_1>0$;
  \item $|z|=1,\;z\neq1,$ and $-1<\beta-b_0-b_1\leq0.$
\end{enumerate}
Function $F$ is also defined by analytic continuation everywhere in $z,\;z\le-1.$
%In other cases the series \eqref{eq:ghf} diverges, 
see 
 \cite[Chap. 9.1]{GR} and \cite{Andrews}.

Therefore, functions $\widehat Q_0$ and $\widehat Q_1$ are defined by formulae \eqref{eq:q0q1=}
and by series \eqref{eq:ghf}, if the starting point $x$ satisfies 
$a_0/\gamma_0\le x<2a_0/\gamma_0-a_1/\gamma_1.$ If $x$ is far from $a_0/\gamma_0,$
analytic continuation is applied.
%%%%%%%%%%%%%%%%%%%%%%%%%%%%%%%%%%%%%
\begin{theo}
If $\la_0>0,$ then a.s.
\[
T(x)<\infty,\qquad x>a/\gamma.
\]
\end{theo}
%%%%%%%%%%%%%%%%%%%
\begin{proof}
Since $b_0(0)=0$ and $b_1(0)=\la_0/\gamma_0+\la_1/\gamma_1,\;\beta_0(0)=\la_0/\gamma_0,$ 
we have 
\[\PP\{T(x)<\infty~|~\ep(0)=0\}=\widehat Q_0(0, x)=F(0; b_1(0); \beta_0(0)+1; z)
\equiv1,
\]
\[
\PP\{T(x)<\infty~|~\ep(0)=1\}=\widehat Q_1(0, x)=F(0; b_1(0); \beta_0(0); z)
\equiv1, 
\]
which give the proof.
 $\hfill\Box$\end{proof} 

From \eqref{eq:q0q1=} one can obtain the moments of the falling time $T(x).$
For simplicity, we give the explicit formulae for the mean values of $T(x),$
when the initial point $x$ is not so far from the attractive band.

\begin{theo}\label{theo:momentsT}
Let $T(x),\;x\ge a_0/\gamma_0,$ be defined by \eqref{def:T}. 

In the following two cases
\begin{enumerate}
  \item $a_0/\gamma_0\le x<2a_0/\gamma_0-a_1/\gamma_1;$
  \item $x=2a_0/\gamma_0-a_1/\gamma_1$ and $\la_1<\gamma_1;$
\end{enumerate}
the mean values of $T(x)$ are given by the series
\begin{align}
\label{eq:E0}
    \EE\{T(x)~|~\ep(0)=0\}&
    =-b_0'(0)\sum_{n=1}^\infty\frac{(\la_0/\gamma_0+\la_1/\gamma_1)_n}{(1+\la_0/\gamma_0)_n}
    \cdot\frac{z^n}{n}
    %=-\frac{1}{\gamma}\sum_{n=1}^\infty\frac{(\beta_0)_n}{n(\delta_0+1)_n}
   % \cdot\left(\frac{a-\gamma x}{2a}\right)^n
    +\frac{1}{\la_0}<\infty,   \\
    \EE\{T(x)~|~\ep(0)=1\}&
      =-b_0'(0)\sum_{n=1}^\infty\frac{(\la_0/\gamma_0+\la_1/\gamma_1)_n}{(\la_0/\gamma_0)_n}
    \cdot\frac{z^n}{n}<\infty.
    % -\frac{1}{\gamma}\sum_{n=1}^\infty\frac{(\beta_0)_n}{n(\delta_0)_n}
 %   \cdot\left(\frac{a-\gamma x}{2a}\right)^n<\infty.
    \label{eq:E1}
\end{align}
Here 
\[
b_0'(0)=\frac{\la_0+\la_1}{\la_0\gamma_1+\la_1\gamma_0}>0
%\frac12\left(\frac{1}{\gamma_0}
%+\frac{1}{\gamma_1}-\frac{1/\gamma_0-1/\gamma_1}{\la_0/\gamma_0+\la_1/\gamma_1}\right)
\]
is the derivative of the minor root, $b_0(q),$ \eqref{def:b}\textup,
and $z=z(x)=\dfrac{a_0/\gamma_0-x}{a_0/\gamma_0-a_1/\gamma_1}$.

If  $x=2a_0/\gamma_0-a_1/\gamma_1$ and $\gamma_1\le\la_1<2\gamma_1,$ then 
only the series  \eqref{eq:E1} for $\EE\{T(x)~|~\ep(0)=0\}$ is finite.

In all other cases,  
the expectations $\EE\{T(x)~|~\ep(0)=0\}$ and $\EE\{T(x)~|~\ep(0)=1\}$ 
follow after analytic continuation of \eqref{eq:q0q1=}.
\end{theo}
\begin{proof}
Since 
\[
b_0(0)=\frac12\left(\beta_0(0)+\beta_1(0)-\sqrt{(\beta_0(0)-\beta_1(0))^2+4\beta_0(0)\beta_1(0)}\right)=0,
\] 
\[
b_1(0)=\frac12\left(\beta_0(0)+\beta_1(0)+\sqrt{(\beta_0(0)-\beta_1(0))^2+4\beta_0(0)\beta_1(0)}\right)=
\frac{\la_0}{\gamma_0}+\frac{\la_1}{\gamma_1}
\]
and by \eqref{eq:ghf}, 
\[
F'_{b_1}(b_0, b_1; \beta_0+1; z)|_{q=0}=0,\qquad
F'_{\beta_0}(b_0, b_1; \beta_0+1; z)|_{q=0}=0,
\]
we have 
\[
\begin{aligned}
    \EE[T(x)~|~\ep(0)=0]=&-\frac{\pd}{\pd q}\left[\widehat Q_0(q, x)\right]|_{q=0}\\
    =\frac{\la_0}{(\la_0+q)^2}|_{q=0}-\frac{\la_0}{\la_0+q}&\cdot\left( b_0'(0)
    F'_{b_0}+b_1'(0)F'_{b_1}+\beta_0'(0)F'_{\beta_0}
    \right)(b_0, b_1; 1+\beta_0;\; z)|_{q=0}
    \\
   =\frac{1}{\la_0}-&b_0'(0)F'_{b_0}(0,\la_0/\gamma_0+\la_1/\gamma_1; 1+\la_0/\gamma_0;\;z) 
   \end{aligned}
\]
with $z=z(x)=\dfrac{a_0/\gamma_0-x}{a_0/\gamma_0-a_1/\gamma_1}.$
Further,
\[
F'_{b_0}(b_0, b_1; 1+\beta_0;\; z) |_{q=0}
=\sum_{n=1}^\infty\frac{(n-1)!(b_1(0))_n}{(1+\beta_0(0))_n}\cdot\frac{z^n}{n!}
=\sum_{n=1}^\infty\frac{(\la_0/\gamma_0+\la_1/\gamma_1)_n}{(1+\la_0/\gamma_0)_n}\cdot\frac{z^n}{n},
\]
if the series converges.

Formula \eqref{eq:E0} follows from
\[
\begin{aligned}
  b_{0}'(q)|_{q=0}  & =\frac12\left(\frac{1}{\gamma_0}+\frac{1}{\gamma_1}
-\frac{(\la_0/\gamma_0-\la_1/\gamma_1)(1/\gamma_0-1/\gamma_1)}
  {\la_0/\gamma_0+\la_1/\gamma_1}
  \right)  \\
    &  =\frac12\frac{2(\la_0+\la_1)/(\gamma_0\gamma_1)}{\la_0/\gamma_0+\la_1/\gamma_1}
    =\frac{\la_0+\la_1}{\gamma_1\la_0+\gamma_0\la_1}
\end{aligned}
\]

Similarly, one can obtain \eqref{eq:E1}.
  $\hfill\Box$\end{proof}
\begin{figure}[ht]
	\centering
	\hspace*{-2.3cm}
	\subfigure[]
 {\includegraphics[scale=0.27]{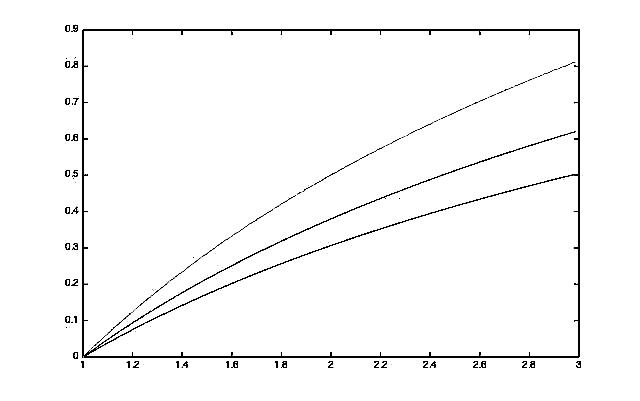}}\quad
	\subfigure[]
	 {\includegraphics[scale=0.295]{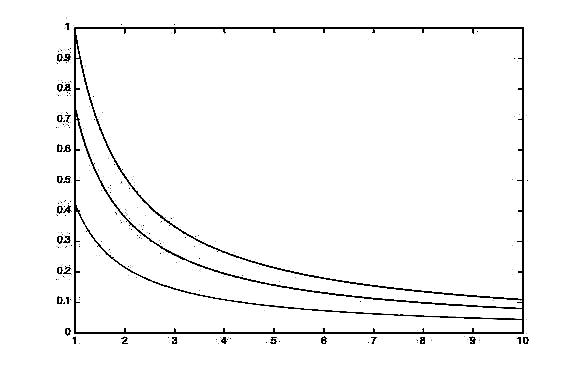}}\\
	 \hspace*{-1.4cm}
	 	\subfigure[]
 {\includegraphics[scale=0.33]{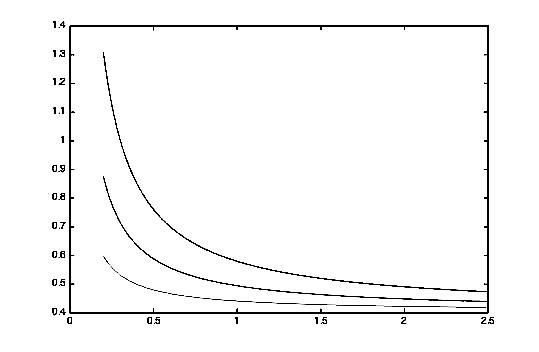}}\;
	\subfigure[]
	 {\includegraphics[scale=0.295]{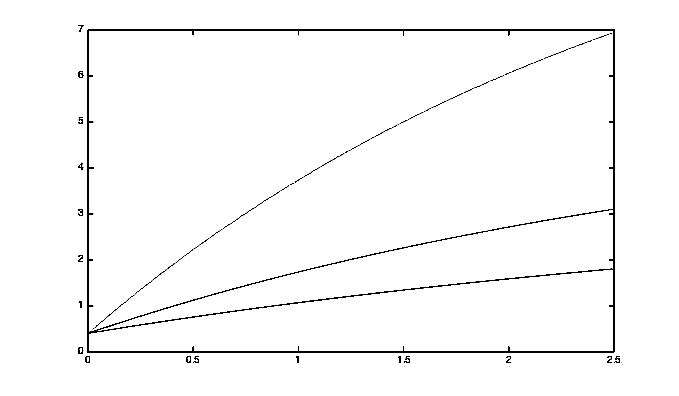}}
		\caption{The expectation $E_1=\EE[T(x)~|~\ep(0)=1]$ 
		in the case $a_0=-a_1=a$ and $\gamma_0=\gamma_1=\gamma,$
		as function of}
		\begin{description}
  \item[(a)] $x,\;1\leq x\leq3,$ with $\la_0=\la_1=1$  for $a=\gamma=1; 2.5; 5$ (from top to bottom);
  \item[(b)] $a,\;1\le a\le 10,$ with $x=2, \;\gamma=a, \la_0=\la_1=1$ for $x=1.5;\;2;\;2.5$ (from top to bottom);
  \item[(c)] $\la_0,\;0.2\le\la_0\leq2.5,$ with $x=2,\;a=\gamma=1$ for $\la_1=0.1;\;0.25;\;0.5$ (from bottom to top);
  \item[(d)]$\la_1,\;0\le\la_1\leq2.5,$ with $x=2,\;a=\gamma=1$ for $\la_0=0.1;\;0.25;\;0.5$ (from top to bottom)
\end{description}
	\label{fig:2}
\end{figure}

The subsequent moments, $\EE[T(x)^n~|~\ep(0)=i],\;n\geq2,\; i\in\{0, 1\},$ can be obtained by sequential differentiation. 

  Some plots are presented in Fig. \ref{fig:2}.

\begin{rem}
Let $X=X(t),$ starts from $x=X(0),\;x<a_1/\gamma_1,$ and 
\[
T_-(x)=\inf\{t~:~X(t)>a_1/\gamma_1\},\qquad x<a_1/\gamma_1,
\]
be the first passage time through the level $x=a_1/\gamma_1.$
The formulae for the expectations of $T_-(x)$  can be easily written by symmetry.
\end{rem}

\begin{rem}
Formulae \eqref{eq:q0q1=} are consistent with some simple reasonable results. 

Let $\la_0=0$.

If $\ep(0)=0,$ then $X(t)=\phi_0(x, t)\;\forall t>0,\;\text{a. s.}$
 and, hence, the process $X$ never crosses the level $a_0/\gamma_0.$ That is, $T(x)=+\infty$.

 If $\ep(0)=1,$ then the process $X=X(t)$ passes through $a_0/\gamma_0$ if and only if 
 there is no switching up to the time $t^*(x),$ \eqref{def:t*}. Therefore\textup{,} 
conditionally (under $\ep(0)=1$)
\[\begin{aligned}
\Big[T(x)=\begin{cases}
      t^*(x),& \text{with probability $\rme^{-\la_1t^*(x)}$ }, \\
      +\infty,& \text{with probability $1-\rme^{-\la_1t^*(x)}$}
\end{cases}~|~\ep(0)=1&\Big].
\end{aligned}\]
In this case,
\begin{equation}\label{eq:q1-0}
\widehat Q_0(q, x)\equiv0,\quad 
\widehat Q_1(q, x)=\exp(-qt^*(x))\cdot\rme^{-\la_1t^*(x)}
=\left(\frac{x-a_1/\gamma_1}{a_0/\gamma_0-a_1/\gamma_1}\right)^{-(\la_1+q)/\gamma_1}.
\end{equation}

The same result is given by  \eqref{eq:q0q1=}: 
if $\la_0=0,$ then, due to \eqref{eq:q0q1=}, we have 
$\widehat Q_0(q, x)\equiv0,$ and $b_0(q),\;b_1(q)$ 
coincide with $\beta_0=q/\la_0,\;\beta_1=(\la_1+q)/\gamma_1$.
Hence,
\[\begin{aligned}
\widehat Q_1(q, x)=&
1+\sum_{n=1}^\infty\frac{(\beta_1(q))_n}{n!}z(x)^n
=\left(1-z(x)\right)^{-\frac{\la_1+q}{\gamma_1}}
=\left(\frac{x-a_1/\gamma_1}{a_0/\gamma_0-a_1/\gamma_1}\right)^{-(\la_1+q)/\gamma_1},
\end{aligned}\]
which coincides with \eqref{eq:q1-0}.
 
Let $\la_1=0.$ 

If the particle begins to move from point $x,\;x>a_0/\gamma_0,$ 
according to the pattern $\phi_1(x, t),$ \eqref{eq:pattern1}, 
then it will  arrive without switching to $a_0/\gamma_0$ at time $t^*(x)$.
It means that
\begin{equation}\label{eq:Q1delta}
Q_1(t, x)=\delta(t-t^*(x)).
\end{equation}
Thus, by \eqref{def:laplaceq0q1} and  \eqref{def:t*}
\[\begin{aligned}
\widehat Q_1(q, x)=&\rme^{-qt^*(x)}
=\left(\frac{x-a_1/\gamma_1}{a_0/\gamma_0-a_1/\gamma_1}\right)^{-q/\gamma_1}
=\left(1-z(x)\right)^{-q/\gamma_1}\\
=&1+\sum_{n=1}^\infty\frac{(b_1)_n}{n!}z(x)^n
\end{aligned}\]
with $b_1=b_1(q)=q/\gamma_1.$ 

 This is repeated by \eqref{eq:q0q1=}
with $b_1=\beta_1=q/\gamma_1$ and $b_0=\beta_0=(\la_0+q)/\gamma_0.$

On the other hand, if the particle begins with the pattern $\phi_0(x, t),$ \eqref{eq:pattern0}, 
then it falls into $a_0/\gamma_0$ after a single switch (at time $\tau$) to the pattern $\phi_1$.  
This means that 
\begin{equation}
\label{eq:hatQ0la0}
\widehat Q_0(q, x)=\EE[\exp(-q\left(\tau+t^*(\phi_0(x, \tau))\right)].
\end{equation}
Since 
  \[\begin{aligned}
t^*(\phi_0(x, \tau))&
=\frac{1}{\gamma_1}\log\frac{ \phi_0(x, \tau)-a_1/\gamma_1}{a_0/\gamma_0-a_1/\gamma_1}
=\frac{1}{\gamma}\log\frac{a_0/\gamma_0-a_1/\gamma_1+(x-a_0/\gamma_0)\rme^{-\gamma_0\tau}}
{a_0/\gamma_0-a_1/\gamma_1}\\
&=\frac{1}{\gamma_1}\log\left(1-z(x)\rme^{-\gamma_0 \tau}\right), \qquad
z=\frac{a_0/\gamma_0- x}{a_0/\gamma_0-a_1/\gamma_1}<0,   
\end{aligned}    \]
equation \eqref{eq:hatQ0la0} becomes
\[
\widehat Q_0(q, x)=\int_0^\infty\la_0\rme^{-(\la_0+q)\tau}
\left(1-z(x)\rme^{-\gamma_0\tau}\right)^{-q/\gamma_1}\rmd\tau\]
\begin{equation}
\label{eq:hatQ0la01}
=\frac{\la_0}{\gamma_0}\int_0^1y^{-1+(\la_0+q)/\gamma_0}\left(1-z(x)y\right)^{-q/\gamma_1}\rmd y.
\end{equation}

Due to the integral representation of Gaussian hypergeometric function \cite[formula 9.111]{GR},
\begin{equation}
\label{eq:hatQ000}
\widehat Q_0(q, x)
=\frac{\la_0}{\la_0+q}F(q/\gamma_1, (\la_0+q)/\gamma_0;1+(\la_0+q)/\gamma_0; z(x)),
\end{equation}
which coincides with the first equation of \eqref{eq:q0q1=} (with $\la_1=0$).
\end{rem}

%%%%%%%%%%%%%%%%%%%%%%%%%%%%%%%%%
\section{The mean and variance of $X(t).$} \label{sec5}
\setcounter{equation}{0}

The marginal distribution of $X(t),$ \eqref{eq:solOU1}, 
can not be so easily written as the distribution of the Gaussian Ornstein-Uhlenbeck process.
In this section we give only a few hints on this matter.

Let $0=\tau_0<\tau_1<\ldots<\tau_n<\ldots$ be the sequence of switching times 
of the of the underlying Markov process $\ep.$
Let $N(t)$ corresponds to the number switchings 
till time $t,\;t>0,$
\[
N(t)=n,\qquad \text{if}\qquad \tau_n\le t<\tau_{n+1}.
\]
Recalling the distribution of the inhomogeneous Poisson process $N(t),$
see \cite[Theorem 2.1]{JAP51}, we have 
\begin{equation}
\label{eq:pijk}
\begin{aligned}
    \pi_{00}(s)=\PP_0\{N(s)\text{ is even}\}&
    =\rme^{-\la_0s}\Big[1+\Psi_0(s,\;(\la_0-\la_1)s)\Big],
    %\sum_{n=1}^\infty\frac{\la_0^n\la_1^n}{(2n)!}s^{2n}\Phi(n, 2n+1;(\la_0-\la_1)s)\right),  
    \\
       \pi_{11}(s)=\PP_1\{N(s)\text{ is even}\}&
       =\rme^{-\la_1s}\Big[1+\Psi_0(s,\;(\la_1-\la_0)s)\Big],
       %\left(1+\sum_{n=1}^\infty\frac{\la_0^n\la_1^n}{(2n)!}s^{2n}    \Phi(n, 2n+1;(\la_1-\la_0)s)\right), 
    \\
    \pi_{01}(s)=\PP_0\{N(s)\text{ is odd}\}&
    =\la_0\rme^{-\la_0s}\Psi_1(s,\;(\la_0-\la_1)s),
    %\sum_{n=1}^\infty\frac{\la_0^n\la_1^{n-1}}{(2n-1)!}s^{2n-1}    \Phi(n, 2n; (\la_0-\la_1)s), 
    \\
        \pi_{10}(s)=\PP_0\{N(s)\text{ is odd}\}&
        =\la_1\rme^{-\la_1s}\Psi_1(s,\;(\la_1-\la_0)s),
        %\sum_{n=1}^\infty\frac{\la_0^{n-1}\la_1^{n}}{(2n-1)!}s^{2n-1}
  %  \Phi(n, 2n; (\la_1-\la_0)s).
\end{aligned}
\end{equation}
where
\begin{equation}
\label{def:Psi}
\Psi_0(t, z)=\sum_{n=1}^\infty\frac{\la_0^n\la_1^n}{(2n)!}t^{2n}\Phi(n, 2n+1; z),\quad
\Psi_0(t, z)=\sum_{n=1}^\infty\frac{\la_0^{n-1}\la_1^{n-1}}{(2n-1)!}t^{2n-1}\Phi(n, 2n; z);
\end{equation}
$\Phi(\cdot, \cdot;\cdot)$ is the confluent hypergeometric function, \cite{Andrews}.

Due to representation \eqref{eq:solOU1}, the mean of $X(t)$ is given by
\begin{equation}
\label{eq:E0X}
\begin{aligned}
\EE_0&[X(t)]=   \EE_0\left[\rme^{-\Gamma(t)}\left(x+\int_0^t\rme^{\Gamma(s)}a_{\ep(s)}\rmd s\right)\right]
=x\psi_0^\Gamma(t)\\
+\int_0^t&
\left[a_0\pi_{00}(s)\EE\left(\rme^{-(\Gamma(t)-\Gamma(s))}~|~\ep(s)=0\right)
+a_1\pi_{01}(s)\EE\left(\rme^{-(\Gamma(t)-\Gamma(s))}~|~\ep(s)=1\right)
\right]\rmd s\\
 &=x \psi_0^\Gamma(t)+a_0\int_0^t\pi_{00}(s)\psi_0^\Gamma(t-s)\rmd s
+a_1\int_0^t\pi_{01}(s)\psi_1^\Gamma(t-s)\rmd s.
\end{aligned}
\end{equation}
Similarly,
\begin{equation}
\label{eq:E1X}
\EE_1[X(t)]=x \psi_1^\Gamma(t)+a_0\int_0^t\pi_{10}(s)\psi_0^\Gamma(t-s)\rmd s
+a_1\int_0^t\pi_{11}(s)\psi_1^\Gamma(t-s)\rmd s.
\end{equation}
Here $\pi_{ik}(\cdot)$ are determined by \eqref{eq:pijk}-\eqref{def:Psi}, and 
the moment generating functions,   
 \[
 \psi_k^\Gamma(t)=\EE_k[\exp(-\int_0^t\gamma_{\ep(s)}\rmd s)],\qquad k\in\{0, 1\},
 \]
 of the telegraph process $\Gamma(t)$ are also known, 
 \[
 \begin{aligned}
   \psi_0^\Gamma(t) &=\rme^{-(\la_0+\gamma_0)t} \left[1+\Psi_0(t,\;(\la_0-\la_1+\gamma_0-\gamma_1)t)
   +\la_0\Psi_1(t,\;(\la_0-\la_1+\gamma_0-\gamma_1)t)\right],  \\
   \psi_1^\Gamma(t) &=\rme^{-(\la_1+\gamma_1)t} \left[1+\Psi_0(t,\;(\la_1-\la_0+\gamma_1-\gamma_0)t)
   +\la_1\Psi_1(t,\; (\la_1-\la_0+\gamma_1-\gamma_0)t)\right], 
\end{aligned}
 \]
 see e.g. \cite[(2.21)]{JAP51}.
 
 %%%%%%%%%%%%%%%%%%%%%%%%%%%%%%%%%%%%%%%% 
 \begin{rem}
 \emph{In the symmetric case, $\la_0=\la_1=\la,\;\gamma_0=\gamma_1=\gamma$ 
 and $a_0=-a_1=a,$} formulae \eqref{eq:E0X}-\eqref{eq:E1X}
 can be simplified. 
 
 Since, $\psi_0^\Gamma(t)=\psi_1^\Gamma(t)=\rme^{-\gamma t}$
 and $\pi_{00}(s)=\pi_{11}(s)=(1+\rme^{-2\la s})/2,\;\pi_{01}(s)=\pi_{10}(s)=(1-\rme^{-2\la s})/2,$
by \eqref{eq:E0X}-\eqref{eq:E1X}
we have
 \begin{equation}
\label{eq:E0Xs}
\EE_0[X(t)]=x\rme^{-\gamma t}+a
\begin{cases}
     \dfrac{\rme^{-2\la t}-\rme^{-\gamma t}}{\gamma -2\la}, 
     & \text{ if $\gamma \neq2\la$}, \\ \\
 t\rme^{-\gamma t},     & \text{if $\gamma =2\la$},
 \end{cases}
\end{equation}
and
\begin{equation}
\label{eq:E1Xs}
\EE_1[X(t)]=x\rme^{-\gamma t}-a
\begin{cases}
     \dfrac{\rme^{-2\la t}-\rme^{-\gamma t}}{\gamma -2\la}, & \text{ if $\gamma \neq2\la$}, \\ \\
 t\rme^{-\gamma t},     & \text{if $\gamma =2\la$}.
\end{cases}
\end{equation}

Further, notice that in the symmetric case, 
\[\begin{aligned}
\EE_0[a_{\ep(s_1)}a_{\ep(s_2)}]=&\EE_1[a_{\ep(s_1)}a_{\ep(s_2)}]
=a^2\frac{1+\rme^{-2\la|s_1-s_2|}}{2}-a^2\frac{1-\rme^{-2\la|s_1-s_2|}}{2}\\
=&a^2\exp(-2\la|s_1-s_2|).
\end{aligned}\]
Hence, %the second moment can be simplified   
\[
\EE\left[\int_0^t\rme^{-\gamma(t-s)}\rmd\TT(t)\right]^2
=a^2\rme^{-2\gamma t}\int_0^t\int_0^t\exp(\gamma (s_1+s_2)-2\la|s_1-s_2|)\rmd s_1\rmd s_2
\]
\begin{equation}
\label{eq:EX2}
=\frac{a^2}{\gamma+2\la}
\begin{cases}
\dfrac{1}{\gamma}-\dfrac{2}{\gamma-2\la}\rme^{-(\gamma+2\la)t}+\dfrac{\gamma+2\la}{\gamma(\gamma-2\la)}\rme^{-2\gamma t},
& \text{if $\gamma \neq2\la$}, \\ \\
 \dfrac{1-\rme^{-2\gamma t}-2\gamma t\rme^{-2\gamma t}}{\gamma},     & \text{if $\gamma=2\la$},
\end{cases}
\end{equation}
which gives the expression for the variance of $X(t),$
\begin{equation}
\label{eq:VarX}
\begin{aligned}
\mathrm{Var}[X(t)]=&\EE\left[\int_0^t\rme^{-\gamma(t-s)}\rmd\TT(t)\right]^2-
\left(\EE\left[\int_0^t\rme^{-\gamma(t-s)}\rmd\TT(t)\right]\right)^2\\
=&a^2
\begin{cases}
     \dfrac{1}{\gamma(\gamma+2\la)}-\dfrac{\rme^{-2\gamma t}}{(\gamma-2\la)^2}
     \left[\rme^{2(\gamma-2\la)t}-\dfrac{8\la}{\gamma+2\la}\rme^{(\gamma-2\la)t}+\dfrac{2\la}{\gamma}\right], 
     & \text{if $\gamma\neq2\la$ }, \\ \\
  \dfrac{1}{2\gamma^2}\left[1-\rme^{-2\gamma t}\left(1+2\gamma t+2\gamma^2t^2\right)\right],
  & \text{if $\gamma=2\la$}.
\end{cases}
\end{aligned}\end{equation}

The limiting behaviour of $X(t)$ is consistent with known results.

As $t\to\infty,$ the limits  are given by
\[
\lim_{t\to\infty}\EE_0[X(t)]=\lim_{t\to\infty}\EE_1[X(t)]=0,
\qquad
\lim_{t\to\infty}\mathrm{Var}[X(t)]=\frac{a^2}{\gamma(\gamma+2\la)}.
\]
On the other hand, under Kac's scaling, $a,\;\la\to\infty,\;a^2/\la\to\sigma^2,$ 
the limits of the expectation 
\begin{equation}\label{eq:EX-kac}
\lim\EE[X(t)]=x\rme^{-\gamma t}\pm\lim a\frac{\rme^{-2\la t}-\rme^{-\gamma t}}{\gamma -2\la}=x\rme^{-\gamma t},
\end{equation}
see \eqref{eq:E0X}-\eqref{eq:E1X},
and the variance 
\begin{equation}\label{eq:varX-kac}
\begin{aligned}
\lim\mathrm{Var}[X(t)]=&\lim a^2\left\{\dfrac{1}{\gamma(\gamma+2\la)}-\dfrac{\rme^{-2\gamma t}}{(\gamma-2\la)^2}
     \left[\rme^{2(\gamma-2\la)t}-\dfrac{8\la}{\gamma+2\la}\rme^{(\gamma-2\la)t}+\dfrac{2\la}{\gamma}\right]\right\}\\
     =&\lim\frac{a^2}{\gamma(\gamma+2\la)}-\frac{\rme^{-2\gamma t}}{\gamma}\lim\frac{2a^2\la}{(\gamma-2\la)^2}
     =\frac{\sigma^2}{2\gamma}\left(1-\rme^{-2\gamma t}\right),
     \end{aligned}\end{equation}

Formulae \eqref{eq:EX-kac}-\eqref{eq:varX-kac}
coincide with the known results for the classical Ornstein-Uhlenbeck process, see e.g. \cite[(4)-(5)]{Maller}.
\end{rem}
%%%%%%%%%%%%%%%%%%%%%%%%%%%%%%%%%%%%%%%%%%%%%%%%%
%%%%%%%%%%%%%%%%%%%%%%%%%%%%%%%%%%%%%%%%%%%%%%%%%
\section{On the joint distribution of $X(t)$ and $N(t).$}\label{sec6}
\setcounter{equation}{0}

Due to technical difficulties, the distribution of the Ornstein-Uhlenbeck 
process with bounded variation cannot be presented explicitly. However, let's sketch it out.
 
Consider the Ornstein-Uhlenbeck process of bounded variation $X=X(t)$ 
 based on the completely symmetric telegraph process $\TT:$ 
 the velocities are $\pm a$ the switching intensities are identical, $\la_0=\la_1=\la$,
 and $\gamma_0=\gamma_1=\gamma$. 
 Let $f_i(y, t; n~|~x),\;n\ge0, i\in\{0, 1\},$
 be the density functions characterising the joint distribution of the particle position
$X(t)$ and the number of the patterns switchings $N(t),$ 
 \begin{equation*}
%\label{def:pi}
f_i(y, t; n~|~x)=\PP\{X(t)\in\rmd y,\;N(t)=n~|~X(0)=x,\;\ep(0)=i\}/\rmd y.
\end{equation*}

By definition, we have 
\begin{equation}
\label{eq:pi0}
f_0(y, t; 0~|~x)=\rme^{-\la t}\delta(y-\phi_0(x, t)),\quad f_1(y, t; 0~|~x)=\rme^{-\la t}\delta(y-\phi_1(x, t)),
\end{equation}
$\phi_0(x, t)=a/\gamma+(x-a/\gamma)\rme^{-\gamma t},\;
\phi_1(x, t)=-a/\gamma+(x+a/\gamma)\rme^{-\gamma t}$.
Further, 
by virtue of \eqref{eq:law01}-\eqref{eq:law10}, functions $f_0(y, t; n~|~x)$ and $f_1(y, t; n~|~x)$ 
satisfy the sequence of coupled integral equations, $n\geq1,$
\begin{align}
\label{eq:phi0}  
f_0(y, t; n~|~x)&=\la\int_0^t\rme^{-\la \tau}f_1(y, t-\tau; n-1~|~\phi_0(x,\tau))\rmd\tau,   \\
 \label{eq:phi1}  
f_1(y, t; n~|~x)&=\la\int_0^t\rme^{-\la \tau}f_0(y, t-\tau; n-1~|~\phi_1(x,\tau))\rmd\tau.
\end{align}

Due to the total symmetry of the underlying process $\TT$, we have the identity in law\textup{:}
\[
[X(t)~|~\ep(0)=0,\;X(0)=x]\stackrel{D}{=}[-X(t)~|~\ep(0)=1,\;X(0)=-x],\qquad t>0.
\]
Moreover, by induction, one can verify the following identities:
for all $n,\;n\geq0,$
\begin{equation}
\label{eq:p1=p0}
f_0(y, t; n~|~x)\equiv f_1(-y, t; n~|~-x),\qquad  t>0.
\end{equation}
Since $\phi_1(-x, t)\equiv -\phi_0(x, t),$ 
for $n=0$ \eqref{eq:p1=p0} follows by definition, see \eqref{eq:pi0}.  
Let  \eqref{eq:p1=p0} be proved for $n-1$. Equations \eqref{eq:phi0}-\eqref{eq:phi1} give 
\[
\begin{aligned}
    f_1(-y, t;n~|~-x)=&\la\int_0^t\rme^{-\la \tau}f_0(-y, t-\tau; n-1~|~\phi_1(-x,\tau))\rmd\tau   \\
                             =&\la\int_0^t\rme^{-\la \tau}f_0(-y, t-\tau; n-1~|~-\phi_0(x,\tau))\rmd\tau \\ 
                             =&\la\int_0^t\rme^{-\la \tau}f_1(y, t-\tau; n-1~|~\phi_0(x,\tau))\rmd\tau
                            = f_0(y, t; n~|~x),
\end{aligned} 
\]
 which proves the result \eqref{eq:p1=p0}.

In order to determine the explicit expressions 
of the density functions $f_0(y, t; n~|~x)$ and $f_1(y, t; n~|~x),$
consider first \eqref{eq:phi0}-\eqref{eq:phi1} with $n=1.$ By \eqref{eq:pi0} we have 
\begin{align}
\label{eq:p0n=1}
    f_0(y, t; 1~|~x)=&\la\rme^{-\la t}\int_0^t\delta(y-\phi_1(\phi_0(x, \tau), t-\tau))\rmd \tau,   \\
    \label{eq:p1n=1}
   f_1(y, t; 1~|~x)=&\la\rme^{-\la t}\int_0^t\delta(y-\phi_0(\phi_1(x, \tau), t-\tau))\rmd \tau.   
\end{align}

Notice that the equations
\begin{align}
\label{eq:tau0}
   y- \phi_1(\phi_0(x, \tau), t-\tau)=0,&   \\
    \label{eq:tau1}
  y-  \phi_0(\phi_1(x, \tau), t-\tau)=0,&  
\end{align}
have the solutions,  $\tau,\;0\leq\tau\leq t,$ if and only if $y\in I(x, t):=[\phi_1(x, t),\;\phi_0(x, t)],$ that is
\begin{equation}
\label{interval}
-\frac{a}{\gamma}+\left(x+\frac{a}{\gamma}\right)\rme^{-\gamma t}=\phi_1(x, t)\leq y\leq
 \phi_0(x, t)=\frac{a}{\gamma}+\left(x-\frac{a}{\gamma}\right)\rme^{-\gamma t}.
\end{equation}

Since 
\begin{align}\label{phi1phi0}
  \phi_1(\phi_0(x, \tau), t-\tau)\equiv& -\frac{a}{\gamma}
+\frac{2a}{\gamma}\rme^{-\gamma(t-\tau)}+\left(x-\frac{a}{\gamma}\right)\rme^{-\gamma t},   \\
\label{phi0]hi1} 
   \phi_0(\phi_1(x, \tau), t-\tau)\equiv& \frac{a}{\gamma}
-\frac{2a}{\gamma}\rme^{-\gamma(t-\tau)}+\left(x+\frac{a}{\gamma}\right)\rme^{-\gamma t},
\end{align}
 see \eqref{eq:pattern0}-\eqref{eq:pattern1},
 the solution of \eqref{eq:tau0}, $\tau=\tau_0(y, t~|~x),$ is given by
\begin{equation}
\label{sol:tau0}
\tau=\tau_0(y, t~|~x)=t+\frac{1}{\gamma}\log\frac{a+\gamma y+(a-\gamma x)\rme^{-\gamma t}}{2a}.
\end{equation}
Similarly, the solution of \eqref{eq:tau1}, $\tau=\tau_1(y, t~|~x)=\tau_0(-y, t~|~-x),$ is 
\begin{equation}
\label{sol:tau1}
\tau=\tau_1(y, t~|~x)=t+\frac{1}{\gamma}\log\frac{a-\gamma y+(a+\gamma x)\rme^{-\gamma t}}{2a}.
\end{equation}
Note that
\[\begin{aligned}
\rme^{-\gamma(t-\tau_0(y, t~|~x))}+\rme^{-\gamma(t-\tau_1(y, t~|~x))}\equiv&
\frac{a+\gamma y+(a-\gamma x)\rme^{-\gamma t}}{2a}+\frac{a-\gamma y+(a+\gamma x)\rme^{-\gamma t}}{2a}\\
\equiv&1+\rme^{-\gamma t}.
\end{aligned}\]

Due to equations \eqref{eq:p0n=1}-\eqref{eq:p1n=1}, \eqref{sol:tau0}-\eqref{sol:tau1} 
and the differential equalities 
\[\begin{aligned}
\rmd_\tau\left[y-\phi_1(\phi_0(x, \tau),\; t-\tau)\right]=&-2a\exp(-\gamma(t-\tau))\rmd\tau,\\
\rmd_\tau\left[y-\phi_0(\phi_1(x, \tau),\; t-\tau)\right]=&2a\exp(-\gamma(t-\tau))\rmd\tau,
\end{aligned}
\qquad 0<\tau<t,
\]
one can obtain the explicit expressions for the density functions
with a single velocity switching,
\[\begin{aligned}
%\label{sol:p0n-1}
    f_0(y, t; 1~|~x)&=\frac{\la\rme^{-\la t}\1_{\{y\in I(x, t)\}}}{a+\gamma y+(a-\gamma x)\rme^{-\gamma t}}
    =\frac{\la\rme^{-(\la-\gamma)t}\1_{\{y\in I(x, t)\}}}{2a}\rme^{-\gamma \tau_0(y, t~|~x)},  \\
 %   \label{sol:p1n-1}
    f_1(y, t; 1~|~x)&=\frac{\la\rme^{-\la t}\1_{\{y\in I(x, t)\}}}{a-\gamma y+(a+\gamma x)\rme^{-\gamma t}}
        =\frac{\la\rme^{-(\la-\gamma)t}\1_{\{y\in I(x, t)\}}}{2a}\rme^{-\gamma \tau_1(y, t~|~x)}.  
\end{aligned}\]

We continue to solve equations \eqref{eq:phi0}-\eqref{eq:phi1} using the following lemma.
\begin{lem}\label{lem}
Let $\tau_0(y, t~|~x)$ and $\tau_1(y, t~|~x)$ be defined by \eqref{sol:tau0}-\eqref{sol:tau1}. 
We have the following identities\textup{:}
\begin{equation}
\label{eq:tildetau2}
\begin{aligned}
\tau_0(y, t-\tau~|~\phi_0(x, \tau))=&\tau_0(y, t~|~x)-\tau,\\
\tau_1(y, t-\tau~|~\phi_1(x, \tau))=&\tau_1(y, t~|~x)-\tau,
\end{aligned}
\qquad
0<\tau<t;
\end{equation}
and
\begin{equation}
\label{eq:tildetau3}
\begin{aligned}
    \tau_0(y, t-\tau~|~\phi_1(x, \tau))=& t-\tau+\frac{1}{\gamma}
    \log\frac{a+\gamma y-(a+\gamma x)\rme^{-\gamma t}+2a\rme^{-\gamma t}\rme^{\gamma\tau}}{2a}, \\
     \tau_1(y, t-\tau~|~\phi_0(x, \tau))=&t-\tau+\frac{1}{\gamma}
    \log\frac{a-\gamma y-(a-\gamma x)\rme^{-\gamma t}+2a\rme^{-\gamma t}\rme^{\gamma\tau}}{2a},  
\end{aligned}
\qquad
0<\tau<t.
\end{equation}
Furthermore, 
\begin{equation}
\label{eq:tilde1}
\begin{aligned}
   y\in I(\phi_0(x, \tau), t-\tau) & \Leftrightarrow 0<\tau\leq\tau_0(y, t~|~x),\\
   y\in I(\phi_1(x, \tau), t-\tau) & \Leftrightarrow \tau_1(y, t~|~x)\leq\tau<t.
    \end{aligned}
\end{equation}
\end{lem}
\begin{proof}
Equalities \eqref{eq:tildetau2}-\eqref{eq:tilde1} can be verified directly by definition.
For instance, by \eqref{sol:tau0} and \eqref{eq:pattern0}-\eqref{eq:pattern1} one can obtain the first identities 
of \eqref{eq:tildetau2} and \eqref{eq:tildetau3}:
\[\begin{aligned}
\tau_0(y, t-\tau~|~\phi_0(x, \tau))=&t-\tau
+\frac{1}{\gamma}\log\frac{a+\gamma y+(a-\gamma\phi_0(x, \tau))\rme^{-\gamma(t-\tau)}}{2a}\\
=&t-\tau
+\frac{1}{\gamma}\log\frac{a+\gamma y+(a-\gamma x)\rme^{-\gamma t}}{2a}\\
=&t-\tau+\tau_0(y, t~|~x)-t=\tau_0(y, t~|~x)-\tau
\end{aligned}\]
and
\[\begin{aligned}
\tau_0(y, t-\tau~|~\phi_1(x, \tau))=&t-\tau
+\frac{1}{\gamma}\log\frac{a+\gamma y+(a-\gamma\phi_1(x, \tau))\rme^{-\gamma(t-\tau)}}{2a}\\
=&t-\tau
+\frac{1}{\gamma}\log\frac{a+\gamma y-(a+\gamma x)\rme^{-\gamma t}+2a\rme^{-\gamma t}\rme^{\gamma\tau}}{2a}.
\end{aligned}\]

Further, $y\in I(\phi_0(x, \tau), t-\tau)$ is equivalent to 
\[\phi_1(\phi_0(x, \tau), t-\tau)\leq y\leq\phi_0(\phi_0(x, \tau), t-\tau)\equiv\phi_0(x, t),\]
see \eqref{interval}. 
 Function $\tau\to\phi_1(\phi_0(x, \tau), t-\tau),$ see \eqref{phi1phi0}, increases.  
Hence, $y\in I(\phi_0(x, \tau), t-\tau)$ is equivalent to $0<\tau<\tau_0(y, t~|~x)$.
Other equalities of the lemma are verified similarly.
 $\hfill\Box$\end{proof} 

Due to Lemma \ref{lem} equations \eqref{eq:phi0}-\eqref{eq:phi1} give
\begin{equation}
\label{eq:phi=2}
\begin{aligned}
    f_0(y, t; 2~|~x)=&\la^2\rme^{-\la t}\int_0^{\tau_0(y, t~|~x)}
    \frac{\rmd\tau}{a-\gamma y+(\gamma x-a)\rme^{-\gamma t}+2a\rme^{-\gamma t}\rme^{\gamma\tau}},   \\
    f_1(y, t; 2~|~x)=&\la^2\rme^{-\la t}\int_{\tau_1(y, t~|~x)}^t
    \frac{\rmd\tau}{a+\gamma y-(\gamma x+a)\rme^{-\gamma t}+2a\rme^{-\gamma t}\rme^{\gamma\tau}},   
\end{aligned}
\end{equation}

Since
\[
\int_a^b\frac{\rmd\tau}{A+B\rme^{\gamma\tau}}=\frac{1}{\gamma A}
\log\left[\frac{A+B\rme^{\gamma a}}{A+B\rme^{\gamma b}}\rme^{\gamma(b-a)}\right],
\]
integrating in \eqref{eq:phi=2}, we get
\[
\begin{aligned}
    f_0(y, t; 2~|~x)=&  \la^2\rme^{-\la t}\frac{\tau_0(y, t~|~x)+\tau_1(y, t~|~x)-t}{a-\gamma y+(\gamma x-a)\rme^{-\gamma t}}, \\
    f_1(y, t; 2~|~x)=&  \la^2\rme^{-\la t}\frac{\tau_0(y, t~|~x)+\tau_1(y, t~|~x)-t}{a+\gamma y-(\gamma x+a)\rme^{-\gamma t}},
\end{aligned}
\]
where $\tau_0(y, t~|~x)$ and $\tau_1(y, t~|~x)$ are determined by \eqref{sol:tau0}-\eqref{sol:tau1}.

Applying Lemma \ref{lem} successively,
 one can obtain a sequence of the formulae for $f_i(\cdot, \cdot; n~|~\cdot),$ 
 which look more and more sophisticated.

\renewcommand{\theequation}{A.\arabic{equation}}
\renewcommand{\thetheo}{A.\arabic{theo}}

\section*{Appendix: the telegraph process}\label{sec2}
\setcounter{equation}{0}
\setcounter{theo}{0}

Let $(\Omega,\;\mathcal F,\;\mathcal F_t,\;\PP)$ be the complete filtered probability space.
Consider the adapted telegraph process $\TT(t),\;t\geq0,$  
with two alternating symmetric velocities $a$ and $-a,\;$ $a>0,$ 
switching with positive intensities $\la_0$ and $\la_1$.  

The joint distribution of $\TT(t)$ and $\ep(t)$ 
can be expressed by means of the (generalised) density functions
\begin{equation*}
%\label{def:pij}
p_{i}^j(x, t):=\PP\{\TT(t)\in\rmd x,\;\ep(t)=j~|~\ep(0)=i\}/\rmd x,\qquad i, j\in\{0, 1\},\quad t>0.
\end{equation*}
The following formulae seem to be generally known, 
 but for the best of my belief, they have never been published.
%%%%%%%%%%%%%%%%%%%%%%%%%%%%%%%%%%%%%%
%%%%%%%%%%%%%%%%%%%%%%%%%%%%%%%%%%%%%%
\begin{theo}\label{theo:distXep}
The density functions $p_{i}^j(x, t),\;i, j\in\{0, 1\},$ are given by
\begin{equation}
\label{eq:distXep}
\begin{aligned}
    p_{0}^0(x, t)&=\rme^{-\la_0t}\delta(x-a_0t)
+\frac{\sqrt{\la_0\la_1}}{a_0-a_1}\sqrt{\frac{\xi}{t-\xi}}
\rme^{-\la_0\xi-\la_1(t-\xi)}I_1(2\sqrt{\la_0\la_1\xi(t-\xi)})\1_{\{0<\xi<t\}},   \\
    p_{1}^1(x, t)&=  \rme^{-\la_1t}\delta(x-a_1t)
    +\frac{\sqrt{\la_0\la_1}}{a_0-a_1}\sqrt{\frac{t-\xi}{\xi}}\rme^{-\la_0\xi-\la_1(t-\xi)}I_1(2\sqrt{\la_0\la_1\xi(t-\xi)})\1_{\{0<\xi<t\}},\\
    p_{0}^1(x, t)&=\frac{\la_0}{a_0-a_1}\rme^{-\la_0\xi-\la_1(t-\xi)}I_0(2\sqrt{\la_0\la_1\xi(t-\xi)})\1_{\{0<\xi<t\}},\\
    p_{1}^0(x, t)&=\frac{\la_1}{a_0-a_1}\rme^{-\la_0\xi-\la_1(t-\xi)}I_0(2\sqrt{\la_0\la_1\xi(t-\xi)})\1_{\{0<\xi<t\}},
\end{aligned}
\end{equation}
where $\xi=(x-a_1t)/(a_0-a_1),\;t-\xi=(a_0t-x)/(a_0-a_1),\;a_1t<x<a_0t.$

Here $I_0$ and $I_1$ are the modified Bessel functions\textup{,}
\[
I_0(z)=1+\sum_{n=1}^\infty\frac{(z/2)^{2n}}{(n!)^2},\qquad
I_1(z)=I_0'(z)=\sum_{n=1}^\infty\frac{(z/2)^{2n-1}}{(n-1)!n!}.
\]
\end{theo}
%%%%%%%%%%%%%%%%%%%%%%%%%%%%%%
\begin{proof}
Let $N(t)$ be the number of velocity switchings in the time interval $[0, t)$.

By virtue of \cite[(4.1.10)-(4.1.11)]{KR}, $p_{0}^0,$ can be represented as 
\begin{equation*}
%\label{eq:pij}
\begin{aligned}
p_{0}^0(x, t)
=&\sum_{n=0}^\infty\PP\{\TT(t)\in\rmd x,\;N(t)=2n~|~\ep(0)=0\}/\rmd x\\
=\rme^{-\la_0t}\delta(x-a_0t)&
+\frac{\exp(-\la_0\xi-\la_1(t-\xi))}{a_0-a_1}
\sum_{n=1}^\infty\frac{\la_0^n\la_1^n}{(n-1)!n!}\xi^n(t-\xi)^{n-1}\1_{\{0<\xi<t\}}\\
=&\rme^{-\la_0t}\delta(x-a_0t)
+\frac{\sqrt{\la_0\la_1}}{a_0-a_1}\sqrt{\frac{\xi}{t-\xi}}I_1(2\sqrt{\la_0\la_1\xi(t-\xi)})\1_{\{0<\xi<t\}}.
\end{aligned}\end{equation*}
see  \cite[formula 8.445]{GR}. 
The remaining equalities of \eqref{eq:distXep} are obtained in the same manner.
 $\hfill\Box$\end{proof} 

The well-known formulae for the (conditional) distribution of $\TT(t)$ follow from \eqref{eq:distXep}:
\[
\begin{aligned}
 p_0(x, t)= \PP\{\TT(t)\in\rmd x~|~\ep(0)=0\}/\rmd x  &= p_0^0(x, t)+p_0^1(x, t),  \\
     p_1(x, t)=   \PP\{\TT(t)\in\rmd x~|~\ep(0)=1\}/\rmd x  & = p_1^0(x, t)+p_1^1(x, t),  
\end{aligned}
\]
 cf 
 \cite{Luisa}, \cite{JAP51} or see in the book by Kolesnik and Ratanov, \cite[(4.1.15)]{KR}.
 
 Similarly, $f(x, t)=p_0^0(x, t)+p_1^0(x, t)$ (and $b(x, t)=p_0^1(x, t)+p_1^1(x, t)$) 
 are the distribution density functions of the moving forward (and backward) particles, cf \cite{O95},
 where these formulae were presented in the symmetric case, $\la_0=\la_1.$

The rest of this section is devoted to a description of the first 
and the second moments of $\TT(t)$ and 
$\TT(t)\1_{\{\ep(t)=j\}},\;j\in\{0, 1\}.$

We will use the following notations
 \[
\EE_i[g(\TT(t))]=\EE[g(\TT(t))~|~\ep(0)=i]=\int_{-\infty}^\infty g(x)p_i(x, t)\rmd x
\]
and
\[
\EE_i^j[g(\TT(t))]=\EE[g(\TT(t))\cdot\1_{\{\ep(t)=j\}}~|~\ep(0)=i]=\int_{-\infty}^\infty g(x)p_i^j(x, t)\rmd x,
\qquad i, j\in\{0, 1\}.
 \] 

%%%%%%%%%%%%%%%%%%%%%%%%%%%%%%%%%%
%%%%%%%%%%%%%%%%%%%%%%%%%%%%%%%%%%
\begin{theo}\label{theo:ET}
Let $a_0=-a_1=a>0.$

For $t\geq0$ 
\begin{align}
\label{eq:E00T}
    \EE_0^0\TT(t)&=a\rme^{-\la_0t}\sum_{n=0}^\infty\frac{\la_0^n\la_1^n}{(2n)!}t^{2n+1}G_n^{(1)}(t),\\
\label{eq:E01T}
      \EE_0^1\TT(t)
   & =a\rme^{-\la_0t}\sum_{n=0}^\infty\frac{\la_0^{n+1}\la_1^n}{(2n+1)!}t^{2n+2}H_n^{(1)}(t),\\
     \label{eq:E10T}
     \EE_1^0\TT(t)& 
     =-a\rme^{-\la_1t}\sum_{n=0}^\infty\frac{\la_0^{n}\la_1^{n+1}}{(2n+1)!}t^{2n+2}H_n^{(1)}(-t)\\
    \label{eq:E11T}
     \EE_1^1\TT(t)&=-a\rme^{-\la_1t}\sum_{n=0}^\infty\frac{\la_0^n\la_1^n}{(2n)!}t^{2n+1}G_n^{(1)}(-t),
\end{align}
and
\begin{align}
\label{eq:E00T2}
\EE_0^0&\TT(t)^2=a^2\exp(-\la_0t)
\sum_{n=0}^\infty\frac{\la_0^n\la_1^n}{(2n)!}t^{2n+2}G_n^{(2)}(t),\\
\label{eq:E01T2}
\EE_0^1&\TT(t)^2=a^2\exp(-\la_0t)
\sum_{n=0}^\infty\frac{\la_0^{n+1}\la_1^n}{(2n+1)!}t^{2n+3}H_n^{(2)}(t),\\
\label{eq:E10T2}
\EE_1^0&\TT(t)^2=a^2\exp(-\la_1t)
\sum_{n=0}^\infty\frac{\la_0^{n}\la_1^{n+1}}{(2n+1)!}t^{2n+3}H_n^{(2)}(-t),\\
\label{eq:E11T2}
\EE_1^1&\TT(t)^2=a^2\exp(-\la_1t)
\sum_{n=0}^\infty\frac{\la_0^n\la_1^n}{(2n)!}t^{2n+2}G_n^{(2)}(-t),
\end{align} 
where
\begin{equation}
\label{def:G1}
G_n^{(1)}(t)=-\frac{2n}{2n+1}\Phi(n+1, 2n+2;  2\beta t)+\Phi(n, 2n+1;  2\beta t),
\end{equation}
\begin{equation}
\label{def:H1}
H_n^{(1)}(t)=-\Phi(n+2, 2n+3; 2\beta t)+\Phi(n+1, 2n+2; 2\beta t)
\end{equation}
and
\begin{equation}
\label{def:G2}
G_n^{(2)}(t)=\frac{2n}{2n+1}\Phi(n+2, 2n+3; 2\beta t)
-\frac{4n}{2n+1}\Phi(n+1, 2n+2; 2\beta t)+\Phi(n, 2n+1; 2\beta t),
\end{equation}
\begin{equation}
\label{def:H2}
H_n^{(2)}(t)=\frac{2n+4}{2n+3}\Phi(n+3, 2n+4; 2\beta t)
-2\Phi(n+2, 2n+3; 2\beta t)+\Phi(n+1, 2n+2; 2\beta t),
\end{equation}
 $2\beta=\la_0-\la_1.$
 
 Here
$\Phi(a, b; z)$ denotes the confluent hypergeometric function,
\[
\Phi(\alpha, \beta; z):=\sum_{n=0}^\infty\frac{(\alpha)_n}{(\beta)_n}\frac{z^n}{n!},
\] 
$(\cdot)_n$ is the Pochhammer symbol\textup{;} 
$(\gamma)_n=\gamma(\gamma+1)\ldots(\gamma+n-1),\;n\geq1,\;(\gamma)_0=1.$
\end{theo}
\begin{proof}
Consider 
\begin{equation*}
%\label{def:mgf}
\begin{aligned}
\psi_i(z, t)=&\EE_i\exp(z\TT(t))=\EE\left[\exp(z\TT(t))~|~\ep(0)=i\right],\\
\psi_i(z, t; n)=&\EE_i\left[\exp(z\TT(t))\1_{\{N(t)=n\}}\right],\qquad n\geq0,
\end{aligned}
\qquad i\in\{0, 1\},\end{equation*}
and 
\begin{equation*}
%\label{def:partial_mgf}
\psi_i^j(z, t)=\EE_i^j\left[\exp(z\TT(t))\right]=\EE_i\left[\exp(z\TT(t))\1_{\{\ep(t)=j\}}\right],\qquad
i, j\in\{0, 1\},
\end{equation*}
corresponding to the moment generating function of $\TT(t).$ 
Notice that
$\psi_i(z, t)=\sum\limits_{n=0}^\infty \psi_i(z, t; n)$
and
\[
\psi_i^i(z, t)=\sum_{n=0}^\infty \psi_i(z, t; 2n),\qquad
\psi_i^{1-i}(z, t)=\sum_{n=0}^\infty \psi_i(z, t; 2n+1),
\quad i\in\{0, 1\}.
\]

The explicit expressions for $\psi_0(z, t; n)$ and $\psi_1(z, t; n)$ 
can be written separately for even and odd $n,\;n\geq0.$ 
Due to \cite[Theorem 2.1]{JAP51}, we have 
\begin{align}
    \psi_0(z, t; 2n)
    \label{eq:psi02n}
&=\frac{\la_0^n\la_1^n}{(2n)!} t^{2n}\Phi\left(n, 2n+1; 2(\beta-az)t\right)\exp(-(\la_0-az)t),\\
      \psi_1(z, t; 2n)
    \label{eq:psi12n}
&=\frac{\la_0^n\la_1^n}{(2n)!} t^{2n}\Phi(n, 2n+1; 2(az-\beta)t)\exp(-(\la_1+az)t), 
\end{align}
\begin{align}
    \psi_0(z, t; 2n+1) 
    \label{eq:psi02n+1}
   &=\frac{\la_0^{n+1}\la_1^n}{(2n+1)!} t^{2n+1}\Phi\left(n+1, 2n+2; 2(\beta-az)t\right)\exp(-(\la_0-az)t),\\
      \psi_1(z, t; 2n+1) 
    \label{eq:psi12n+1}
&=\frac{\la_0^n\la_1^{n+1}}{(2n+1)!} t^{2n+1}      \Phi(n+1, 2n+2; 2(az-\beta)t)\exp(-(\la_1+az)t).
\end{align}

Formulae \eqref{eq:psi02n}-\eqref{eq:psi12n+1} directly give the desired
 result \eqref{eq:E00T}-\eqref{eq:E10T2}.
For instance, by differentiating in \eqref{eq:psi02n} we have
\[
\EE_0^0[\TT(t)]=\sum_{n=0}^\infty\frac{\pd\psi_0(z, t; 2n)}{\pd z}|_{z=0}
\]
\[
=a\rme^{-\la_0t}\sum_{n=0}^\infty\frac{\la_0^n\la_1^n}{(2n)!}t^{2n+1}
\Big[
-2\Phi'(n, 2n+1; 2\beta t)+\Phi(n, 2n+1; 2\beta t)
\Big]
\]
and
\[
\EE_0^0[\TT(t)^2]=\sum_{n=0}^\infty\frac{\pd^2\psi_0(z, t; 2n)}{\pd z^2}|_{z=0}
\]
\[
=a^2\rme^{-\la_0t}\sum_{n=0}^\infty\frac{\la_0^n\la_1^n}{(2n)!}t^{2n+2}
\Big[
4\Phi''(n, 2n+1; 2\beta t)-4\Phi'(n, 2n+1; 2\beta t)+\Phi(n, 2n+1; 2\beta t)
\Big].
\]
The following known identities,
\[
\Phi'(\alpha, \beta; z)=
\dfrac{\rmd\Phi}{\rmd z}(\alpha, \beta; z)=\dfrac{\alpha}{\beta}\Phi(\alpha+1, \beta+1;z)\]
and \[\Phi''(\alpha, \beta; z)=\dfrac{\alpha(\alpha+1)}{\beta(\beta+1)}\Phi(\alpha+2, \beta+2;z),\]
 see \cite[formula 9.213]{GR}, give the result, \eqref{eq:E00T}, \eqref{def:G1} and \eqref{eq:E00T2}, \eqref{def:G2}.
 Formulae \eqref{eq:E01T} and  \eqref{eq:E01T2} can be obtained similarly from \eqref{eq:psi02n+1}.
 
The remaining formulae of the theorem can be derived 
from \eqref{eq:E00T}-\eqref{eq:E01T} and \eqref{eq:E00T2}-\eqref{eq:E01T2}
by symmetry:
formula  \eqref{eq:E11T} follows from \eqref{eq:E00T};  \eqref{eq:E10T} follows from \eqref{eq:E01T};
 \eqref{eq:E11T2} follows from \eqref{eq:E00T2};  \eqref{eq:E10T2} follows from \eqref{eq:E01T2}
 after replacements $a\rightarrow -a$ and $\la_0\leftrightarrow \la_1.$
 $\hfill\Box$\end{proof}
 
 Formulae  \eqref{eq:E00T}-\eqref{eq:E11T2} permit us to evaluate the covariance between $\TT(t)$ and $\TT(s).$
 %%%%%%%%%%%%%%%%%%%%%%%%%%%%%%%%%%%%%%%%%%%%%%%%%%%%%
 \begin{theo}\label{theo:cov}
 For $t>s>0$ 
 \begin{align}
\label{eq:cov0}
\EE_0\TT(t)\TT(s)=& 
\EE_0\TT(t-s)\cdot\EE_0^0\TT(s)+\EE_1\TT(t-s)\cdot\EE_0^1\TT(s)+\EE_0[\TT(s)^2],\\
\label{eq:cov1}
\EE_1\TT(t)\TT(s)=&
\EE_0\TT(t-s)\cdot\EE_1^0\TT(s)+\EE_1\TT(t-s)\cdot\EE_1^1\TT(s)+\EE_1[\TT(s)^2],
\end{align}
where $\EE_i^j[\TT(s)],$ $\EE_i[\TT(t-s)]$ and $\EE_i[\TT(s)^2],\;i, j\in\{0, 1\},$ are given by
 \eqref{eq:E00T}-\eqref{eq:E11T2}.
 \end{theo}
 \begin{proof}
Notice that
 $%\[
 \EE_0\TT(t)\TT(s)=\EE_0\left[(\TT(t)-\TT(s))\cdot\TT(s)\right]+\EE_0[\TT(s)^2].
$% \]

Due to persistence and time-homogeneity of the process $\TT,$ see \eqref{eq:persistence},
 \[
 \begin{aligned}
     &\EE_0\left[(\TT(t)-\TT(s))\cdot\TT(s)\right]=   \\
   &\EE[(\TT(t)-\TT(s)~|~\ep(s)=0]\cdot\EE_0^0\TT(s)
 +\EE[(\TT(t)-\TT(s)~|~\ep(s)=1]\cdot\EE_0^1\TT(s) \\ 
 &=\EE_0\TT(t-s)\cdot\EE_0^0\TT(s)+\EE_1\TT(t-s)\cdot\EE_0^1\TT(s),
\end{aligned}
 \]
 which gives \eqref{eq:cov0}. Formula \eqref{eq:cov1} follows similarly.
 $\hfill\Box$\end{proof}
 
 %%%%%%%%%%%%%%%%%%%%%%%%%%%%%%%%%%%%%
 \begin{rem}
In the symmetric case  $\la_0=\la_1=\la>0,$
the results of Theorem \textup{\ref{theo:ET}} 
(formulae \eqref{eq:E00T}-\eqref{eq:E11T2}) and \eqref{eq:cov0}-\eqref{eq:cov1}
look much simpler, cf \cite{sectorul}.

Since $\beta=0$ and $\Phi(\cdot, \cdot; 0)=1,$ we have %
\[
G_n^{(1)}(t)|_{\beta=0}=G_n^{(2)}(t)|_{\beta=0}\equiv\frac{1}{2n+1},\qquad
H_n^{(1)}(t)|_{\beta=0}\equiv0,\qquad H_n^{(2)}(t)|_{\beta=0}\equiv\frac{1}{2n+3}.
\]
Therefore\textup{,} for the symmetric case\textup{,}  the first moments \eqref{eq:E00T}-\eqref{eq:E10T} are given by
\begin{equation}
\label{eq:EijT-symm}
\begin{aligned}
 \EE_0^0\TT(t)&=- \EE_1^1\TT(t)=at\rme^{-\la t}\sum_{n=0}^\infty\frac{(\la t)^{2n}}{(2n+1)!}
 =at\rme^{-\la t}\frac{\sinh{\la t}}{\la t} 
 =\frac{a}{2\la}\left(1-\rme^{-2\la t}\right),  \\
  \EE_0^1\TT(t)&=\EE_1^0\TT(t)=0,
\end{aligned}
\end{equation}
and
$%\begin{equation}
%\label{eq:ET}
\EE_0\TT(t)=-\EE_1\TT(t)=\dfrac{a}{2\la}\left(1-\rme^{-2\la t}\right).
$%\end{equation}

The second moments are given by 
\begin{equation*}
%\label{eq:EijT2-symm}
\begin{aligned}
  \EE_0^0\TT(t)^2  &=\EE_1^1\TT(t)^2\\
  &=(at)^2\rme^{-\la t}\sum_{n=0}^\infty\frac{(\la t)^{2n}}{(2n+1)!}=(at)^2\rme^{-\la t}\frac{\sinh{\la t}}{\la t}
  =\frac{a^2t}{2\la} \left(1-\rme^{-2\la t}\right),  \\
    \EE_0^1\TT(t)^2  &=\EE_1^0\TT(t)^2\\ 
    &=(at)^2\rme^{-\la t}\sum_{n=0}^\infty\frac{(\la t)^{2n+1}}{(2n+1)!(2n+3)}
    =(at)^2\rme^{-\la t}\left(\frac{\sinh z-z}{z}\right)'|_{z=\la t}\\
    &=\frac{a^2t}{2\la}\left(1+\rme^{-2\la t}\right)-\frac{a^2}{2\la^2}\left(1-\rme^{-2\la t}\right),
\end{aligned}
\end{equation*}
and, by summing we have
\begin{equation}\label{eq:E0T2-symm}
\EE_0\TT(t)^2=  \EE_0^0\TT(t)^2+  \EE_0^1\TT(t)^2    
=\frac{a^2}{2\la^2}\left(\rme^{-2\la t}-1+2\la t\right)=\EE_1\TT(t)^2.
\end{equation}

Formulae \eqref{eq:EijT-symm} and \eqref{eq:E0T2-symm} are consistent 
with known results, see e.g. \cite[(4.2.24)]{KR}. 

By \eqref{eq:cov0}-\eqref{eq:cov1}, \eqref{eq:EijT-symm} and \eqref{eq:E0T2-symm},
\[\begin{aligned}
&\EE_0[\TT(t)\TT(s)]=\EE_1[\TT(t)\TT(s)]\\
&=\frac{a}{2\la}\left(1-\rme^{-2\la(t-s)}\right)\cdot\frac{a}{2\la}\left(1-\rme^{-2\la s}\right)
+0+\frac{a^2}{2\la^2}\left(\rme^{-2\la s}-1+2\la s\right)
\end{aligned}\]
 \begin{equation}
\label{eq:cov0-symm}
=\frac{a^2}{4\la^2}\left[4\la s-(1+\rme^{-2\la(t-s)})(1-\rme^{-2\la s})\right],
\end{equation}
and 
the covariance becomes
 \begin{equation*}\begin{aligned}
%\label{eq:cov1-symm}
\mathrm{cov}(\TT(t),\;\TT(s))=&\EE[\TT(t)\cdot\TT(s)]-\EE[\TT(t)]\cdot\EE[\TT(s)]\\
=&\frac{a^2}{2\la^2}
\left[2\la s-1+\rme^{-2\la s}+\rme^{-2\la t}-\frac12\left(\rme^{-2\la(t-s)}+\rme^{-2\la(t+s)}\right)\right].
\end{aligned}\end{equation*}

It is known that under Kac\textup{'}s scaling\textup{,} $a, \la\to\infty,\;a^2/\la\to\sigma^2,$ 
see \cite{Kac,KR,STAPRO82}\textup{,}
the symmetric telegraph process  $\TT(t)$ converges to Brownian motion $\sigma W(t).$
Formulae  \eqref{eq:EijT-symm}, \eqref{eq:E0T2-symm} and \eqref{eq:cov0-symm} consist with this convergence:
under this scaling
 we have
 \begin{itemize}
  \item  by  \eqref{eq:EijT-symm},
  \[
\lim\EE_0[\TT(t)]=\lim\EE_1[\TT(t)]=0;
\]
  \item by \eqref{eq:E0T2-symm}, 
\[
\lim\EE_0[\TT(t)^2]=\lim\EE_1[\TT(t)^2]=\sigma^2t,
\]
  \item by \eqref{eq:cov0-symm}, 
\[
\lim\EE_0[\TT(t)\TT(s)]=\lim\EE_1[\TT(t)\TT(s)]=\sigma^2s,\qquad s\leq t.
\]
\end{itemize}
 \end{rem} 

\begin{rem}
Notice that the \textup{``}general\textup{"}  telegraph process $\TT(t),\;t\geq0,$ with two alternating velocities 
$a_0$ and $a_1,\;a_0>a_1,$ can be reduced to the symmetric case\textup{:}
\begin{equation*}
%\label{eq:symm-to}
\TT(t)\stackrel{D}{=}(a_0+a_1)t/2+\TT_{\pm a}^{\mathrm{sym}}(t),
\end{equation*}
where $\TT_{\pm a}^{\mathrm sym}(t)$ is the telegraph process with symmetric velocities $\pm a,\;a=(a_0-a_1)/2.$
Therefore\textup{,} without loss of generality, only a ``symmetric" process  $\TT(t),\;t\geq0,$ can be studied.
\end{rem}

\section{Conclusion}
 
The Ornstein-Uhlenbeck process of bounded variation is introduced and discussed in detail.
The definition is based on a version of  Langevin equation when Brownian motion
is replaced by a telegraph process.
This process has an unusual feature: for a finite time,
it falls into a certain interval and remains there forever. 

We study the distribution of this falling time.
The mean and variance of $X(t)$ are also presented explicitly.

In Appendix, we present several new assertions related 
to joint distribution of the telegraph particle position and the current velocity state. 
 
   %%%%%%%%%%%%%%%%%%%%%%%%%%%%%%%%%%%%%%%%%%%

\end{document}